\documentclass[twoside]{article}

\usepackage[accepted]{aistats2023}

\usepackage{amsmath,amssymb,amsthm}
\usepackage[round]{natbib}

\usepackage{xcolor}
\usepackage{tikz}
\usetikzlibrary{calc}
\usepackage[algo2e,ruled,linesnumbered,noend]{algorithm2e}
\usepackage{xfrac}
\usepackage{enumitem}
\usepackage[colorlinks=true]{hyperref}

\SetKwInput{KwInput}{Input}                
\SetKwInput{KwOutput}{Output}              

\newcommand{\argmin}{\operatornamewithlimits{\arg\!\min}}  
\newcommand{\C}{\mathcal{C}}  
\newcommand{\E}{\mathop{\mathbb{E}}}  
\newcommand{\prob}{\mathbb{P}}  

\newcommand{\R}{\mathbb{R}}  
\newcommand{\supp}{\operatorname{supp}}  
\newcommand{\X}{\mathcal{X}}
\newcommand{\Y}{\mathcal{Y}}
\newcommand{\Z}{\mathcal{Z}}

\renewcommand{\bar}{\overline}  
\renewcommand{\hat}{\widehat}  
\renewcommand{\tilde}{\widetilde}  
\renewcommand{\epsilon}{\varepsilon}  

\newtheoremstyle{myThm}   
     {\topsep}                          
     {\topsep}                          
     {\itshape}                         
     {}                                      
     {\sffamily\bfseries}           
     {.}                                     
     {.5em}                              
     {}                                      

\newtheoremstyle{myRem}   
     {\topsep}                          
     {\topsep}                          
     {}                         
     {}                                      
     {\sffamily\bfseries}           
     {.}                               
     {.5em}                              
     {}                                      
     
\newtheoremstyle{myDef}   
     {\topsep}                          
     {\topsep}                          
     {\itshape}                         
     {}                                      
     {\sffamily\bfseries}           
     {.}                                 
     {.5em}                              
     {}                                      

\newcounter{thm}
\theoremstyle{myThm}
\newtheorem{theorem}[thm]{Theorem}
\newtheorem{lemma}[thm]{Lemma}
\newtheorem{proposition}[thm]{Proposition}

\newtheorem{innercustomthm}{Theorem}
\newenvironment{customthm}[1]
  {\renewcommand\theinnercustomthm{#1}\innercustomthm}
  {\endinnercustomthm}

\theoremstyle{myRem}
\newtheorem{remark}[thm]{Remark}

\theoremstyle{myDef}
\newtheorem{definition}[thm]{Definition}

\begin{document}

\twocolumn[
    \aistatstitle{Nonparametric Indirect Active Learning}
    \aistatsauthor{ Shashank Singh }
    \aistatsaddress{ Max Planck Institute for Intelligent Systems, T\"ubingen, Germany }
]

\begin{abstract}
    Typical models of active learning assume a learner can directly manipulate or query a covariate $X$ to study its relationship with a response $Y$. However, if $X$ is a feature of a complex system, it may be possible only to indirectly influence $X$ by manipulating a control variable $Z$, a scenario we refer to as Indirect Active Learning. Under a nonparametric fixed-budget model of Indirect Active Learning, we study minimax convergence rates for estimating a local relationship between $X$ and $Y$, with different rates depending on the complexities and noise levels of the relationships between $Z$ and $X$ and between $X$ and $Y$. We also derive minimax rates for passive learning under comparable assumptions, finding in many cases that, while there is an asymptotic benefit to active learning, this benefit is fully realized by a simple two-stage learner that runs two passive experiments in sequence.
\end{abstract}

\section{INTRODUCTION}

Traditional models of active learning and experimental design assume a learner can directly manipulate, query, or design a covariate $X$ in order to probe its influence on a response $Y$. In many cases, however, a learner's ability to \emph{control} $X$ may be much more limited that its ability to \emph{observe} $X$. A well-studied example is the case of treatment noncompliance, in which an experimenter can prescribe a treatment to a participant (and later ask whether they took the prescribed treatment), but not force them to take the treatment~\citep{zelen1979new,zelen1990randomized,pearl1995causal}. More generally, the covariate $X$ might be the result of a complex process that the learner can influence and measure but not completely control. For example, genetic modifications intended to target specific genes often have limited precision, causing unintended random or systematic (but measurable) changes in other off-target genes and making it challenging to measure the effect of the intended modification~\citep{cellini2004unintended,hendel2015quantifying,wang2019off}. Similarly, it is often unclear whether manipulations conducted in psychological experiments reliably induce a desired psychological state in a participant, necessitating the use of manipulation checks~\citep{festinger1953laboratory,lench2014alternative}. In the most general case, studied here, the learner's influence over $X$ might be governed by an unknown blackbox function and subject to noise.

This paper analyzes nonparametric regression under an \emph{Indirect Active Learning} model, illustrated in Figure~\ref{fig:indirect_learning_model}, in which, rather than specifying the covariate $X$, the learner specifies a control variable $Z$ and then observes both a resulting covariate value $X = g(Z) + \sigma_X \epsilon_X$ and a resulting response $Y = f(X) + \sigma_Y \epsilon_Y$. Here, $f$ and $g$ are unknown functions, $\epsilon_X$ is unobserved centered additive \emph{covariate noise} of level $\sigma_X \geq 0$, and $\epsilon_Y$ is unobserved centered additive \emph{response noise} of level $\sigma_Y \geq 0$. This generalizes the usual setup of (homoskedastic) regression, in which $g(z) = z$, $\epsilon_X = 0$, $f(x) = \E[Y|X = x]$, $\epsilon_Y = Y - \E[Y|X = x]$~\citep{wasserman2006all,tsybakov2008introduction}.
\begin{figure}[htb]
    \centering
    \begin{tikzpicture}
        \node[shape=circle,draw=black] (Z) at (0,0) {$Z$};
        \node[shape=circle,draw=black] (E1) at (2,-1.3) {$\epsilon_X$};
        \node[shape=circle,draw=black] (X) at (2,0) {$X$};
        \node[shape=circle,draw=black] (E2) at (4,-1.3) {$\epsilon_Y$};
        \node[shape=circle,draw=black] (Y) at (4,0) {$Y$};
        \path [->](Z) edge node[below] {$g$} (X);
        \path [->](E1) edge node[left] {$+$} (X);
        \path [->](X) edge node[below] {$f$} (Y);
        \path [->](E2) edge node[left] {$+$} (Y);
        \draw[thick] ($(Z.north west)+(-0.25,0.25)$) rectangle ($(E2.south east)+(0.25,-0.25)$);
    \end{tikzpicture}
    \caption{Graphical model of Indirect Active Learning.}
    \label{fig:indirect_learning_model}
\end{figure}
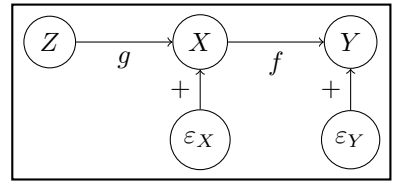
The learner repeats this process iteratively, using past observations of $(X,Y,Z)$ to select a new value for the control variable $Z$ in each iteration, until a prespecified learning budget has been met.
The learner's goal is to estimate the expected response $f(x_0)$ corresponding to a prespecified target covariate value $x_0$.\footnote{Our results extend easily to estimating a contrast, i.e., a difference $f(x_1) - f(x_0)$ at two covariate values $x_0$ and $x_1$.}
To do this efficiently, the learner must query ``informative'' values of $Z$, i.e., those such that $g(Z) + \sigma_X \epsilon_X$ is likely to be near to $x_0$. Thus, in contrast to traditional active learning, which involves only exploitation (selecting the most informative values of $X$), Indirect Active Learning involves an exploration-exploitation trade-off, as the learner must both search for the most informative choices of $Z$ and query these choices repeatedly to mitigate noise.

The present paper begins to characterize the statistical properties of optimal learning under this Indirect Active Learning model. Specifically, in Section~\ref{sec:active_case}, we study minimax convergence rates under a nonparametric model of Indirect Active Learning with a fixed budget, obtaining different rates depending on the complexities of the functions $f$ and $g$ and the noise levels $\sigma_X$ and $\sigma_Y$. Our upper bounds utilize an algorithm (Algorithm~\ref{alg:active_learning_algorithm} in Section~\ref{sec:active_case}) consisting of two passive learning stages, a ``pure exploration'' stage followed by a ``pure exploitation'' stage. We derive high-probability upper bounds on the error of this algorithm as well as minimax lower bounds (on the error of any algorithm) that match our upper bounds in most, though not all, cases.
Because passive experiments are typically much easier to carry out (and parallelize) than sequential active experiments, the optimality of this simple two-stage algorithm in many settings has important practical implications for optimal experimental design.

To understand the specific benefits of active learning under indirect control, Section~\ref{sec:passive_case} derives minimax rates for passive learning under comparable assumptions, which differ from those of typical nonparametric regression. In the passive case, the standard $k$-nearest neighbor ($k$-NN) regressor turns out to be minimax rate-optimal. However, in many cases, the two-stage learner described above asymptotically outperforms the $k$-NN regressor, demonstrating an advantage of (even a very limited form of) active learning over passive learning.

Finally, Section~\ref{sec:experimental_results} presents experiments validating our theoretical results and comparing the performance of active and passive learning in both simple synthetic data and in a more complex epidemiological forecasting application.



\section{NOTATION \& DEFINITIONS}
\label{sec:notation}

For positive real sequences $\{a_n\}_{n = 1}^\infty$ and $\{b_n\}_{n = 1}^\infty$, we write $a_n \lesssim b_n$ if $\limsup_{n \to \infty} a_n/b_n < \infty$, $a_n \gtrsim b_n$ if $b_n \lesssim a_n$, and $a_n \asymp b_n$ if both $a_n \gtrsim b_n$ and $a_n \lesssim b_n$. When ignoring polylogarithmic factors, we write $a_n \lesssim_{\log} b_n$ (to mean $a_n \lesssim b_n (\log b_n)^p$ for some $p > 0$), etc. For positive integers $n$, $[n] = \{1,...,n\}$ is the set of positive integers at most $n$. For $x \in \mathbb{R}$, $\lfloor x \rfloor = \max \{ z \in \mathbb{Z} : z \leq x\}$ is the greatest integer at most $x$.
For $x \in \R^d$, $\|x\| = \max_j |x_j|$ and $\|x\|_2 = \sqrt{\sum_j x_j^2}$ are $\ell_\infty$- and $\ell_2$-norms of $x$.

We now define three regularity conditions needed for our learning problem. The first is \emph{local H\"older smoothness}, used to measure the complexities of the functions $f$ and $g$:
\begin{definition}[Local H\"older Seminorm, Space, and Ball]
    For any $s \in (0, 1]$ and metric spaces $(\Z, \rho_\Z)$ and $(\X, \rho_\X)$, the \emph{local H\"older seminorm} $\|\cdot\|_{\C^s(\Z, z_0; \X)}$ around a point $z_0 \in \Z$ is defined for $f : \Z \to \X$ by
    \[\|f\|_{\C^s(\Z, z_0; \X)} := \sup_{z \neq z_0 \in \Z} \frac{\rho_\X \left( f(z), f(z_0) \right)}{\rho_\Z^s(x, y)},\]
    the \emph{local H\"older space} $\C^s(\Z, z_0; \X)$ around $z_0 \in \Z$ is
    \[\C^s(\Z, z_0; \X) := \left\{ f : \Z \to \X \text{ s.t. } \|f\|_{\C^s(\Z, z_0; \X)} < \infty \right\}.\]
    For $L \geq 0$, the \emph{local H\"older ball} $\C^s(\Z; \X; L)$ is
    \[\C^s(\Z, z_0; \X; L) := \left\{f : \Z \to \X \text{ s.t. } \|f\|_{\C^s(\Z, x_0; \X)} \leq L \right\}.\]
    \label{def:local_holder}
\end{definition}
\emph{Local} H\"older continuity is much weaker than the more common notion \emph{global} H\"older continuity (defined in Appendix~\ref{app:passive_lower_bounds}).
To prove the most general possible bounds, our upper bounds assume only local H\"older continuity around particular points, while matching lower bounds are given under global H\"older continuity.

We next provide a definition of \emph{local dimension} that will be used to measure intrinsic complexities of the random variables $Z$ and $\epsilon_X$:
\begin{definition}[Local Dimension]
    A random variable $X$ taking values in a metric space $(\X, \rho)$ has \emph{local dimension} $d$ around a point $x \in \X$ if
    \[\liminf_{r \downarrow 0} \frac{\prob \left[ \rho(X, x) \leq r \right]}{r^d} > 0.\]
    \label{def:local_assumption}
\end{definition}
This notion of local dimension can be thought of as a local version of Frostman Dimension~\citep{frostman1935potentiel}, and is closely related to the Hausdorff and (lower) box counting or Minkowski dimensions of the underlying space $\X$ (see Corollary 4.12 of \citet{falconer2004fractal} and Chapter 8 of \citet{mattila1999geometry}).

Finally, we state a \emph{sub-Gaussian tail condition}, used to control the noise variables $\epsilon_X$ and $\epsilon_Y$:
\begin{definition}[Sub-Gaussian Random Variable]
    For any integer $d > 0$, an $\R^d$-valued random variable $\epsilon$ is said to be sub-Gaussian if, for all $t \in \R^d$,
    $\E \left[ e^{\langle \epsilon, t \rangle} \right] \leq e^{\|t\|_2^2/2}$.
    \label{def:sub_gaussain}
\end{definition}
Sub-Gaussianity is usually defined with a parameter $\sigma \geq 0$ indicating the scale of $\epsilon$. In this work, to be consistent with other assumptions (namely, the local dimension assumption on $\epsilon_X$), we separate the scale parameter $\sigma$, explicitly writing $\sigma \epsilon$ as needed.

\section{PROBLEM SETUP}
\label{subsec:assumptions}

We now describe our passive and active learning problems.

\paragraph{Passive Setting}
In the passive setting, for some known $x_0 \in \X \subseteq \R^d$, the learner must estimate $f(x_0)$ from pre-collected IID data $\{(X_i,Y_i,Z_i)\}_{i = 1}^n \in \left( \X \times \R \times \Z \right)^n$, assuming:
\begin{enumerate}[nosep]
    \item[\textbf{A1)}] Data are drawn from the model $X = g(Z) + \sigma_X \epsilon_X$, $Y = f(X) + \sigma_Y \epsilon_Y$, as illustrated in Figure~\ref{fig:indirect_learning_model}.
    \item[\textbf{A2)}] The marginal distribution $P_Z$ of $Z$ has dimension $d_Z$ around some $z_0 \in \Z$ with $g(z_0) = x_0$.
    \item[\textbf{A3)}] The distribution of $\epsilon_X$ has dimension $d_X$ around $0$.
    \item[\textbf{A4)}] $g$ and $f$ lie in local H\"older balls $\C^{s_g}(\Z, z_0; \X; L_g)$ and $\C^{s_f}(\X, x_0; \R; L_f)$, respectively.
    \item[\textbf{A5)}] $\epsilon_X$ and $\epsilon_Y$ are sub-Gaussian.
\end{enumerate}
Note that our bounds will depend on the intrinsic dimensions $d_Z$ of $Z$ and $d_X$ of $\epsilon_X$, but not on the extrinsic dimension $d$ of $X$.
Assumptions A2), A3), A4), and A5) are quite mild, comparable to the weakest assumptions typically made in nonparametric regression. The strongest assumption is the heteroscedastic additive error model in A1), but this is primarily for convenience and can also be weakened substantially; for example, our upper bounds hold even if $\sigma_Y$ is a sufficiently a smooth function of $X$ around $x_0$; see Remark~\ref{remark:heteroscedastic_errors_app} in the Appendix for details.

\paragraph{Active Setting}
In the active setting, at each time step $i$, the learner must select a new sample $Z_i$ based on the past data $\{(X_j,Y_j,Z_j)\}_{j = 1}^{i - 1}$. The learner may also utilize the ``prior'' distribution $P_Z$ of $Z$ from the passive case.\footnote{Although not a prior on $z_0$ in the typical Bayesian sense, $P_Z$ plays a similar role in that the estimators' performance will depend on how concentrated $P_Z$ is around $z_0$.}
After selecting $Z_i$, the learner observes the covariate-response pair $(X_i, Y_i)$, and the process repeats. After $n$ iterations of this process, the learner must estimate $f(x_0)$ (for some $x_0 \in \X$ known in advance), using the dataset $\{(X_i,Y_i,Z_i)\}_{i = 1}^n$. The learner may also assume \textbf{A1)}-\textbf{A5)} from the passive setting.

\section{RELATED WORK}
\label{sec:related_work}

Several models of active learning have been previously considered~\citep{settles2009active}, including Query Synthesis~\citep{angluin1988queries,wang2015active} or Adaptive Sampling~\citep{castro2005faster}, in which the learner can request a label for any desired query $x \in \X$, Pool-Based Sampling~\citep{lewis1995sequential,hanneke2014theory}, in which the learner can request a label for any $x \in S$ given a finite ``data pool'' $S \subseteq \X$, and Stream-Based Selective Sampling~\citep{atlas1989training}, in which a learner can request labels for individual data points as they appear sequentially.
Although assumed in many theoretical analyses~\citep{burnashev1974interval,sung1994active,castro2005faster}, Query Synthesis can often be impractical, as generating arbitrary $x \in \X$ to query may be difficult (although progress is being made on this front, for example using modern generative models~\citep{zhu2017generative}).
Indirect Active Learning can be viewed as a significant relaxation of Query Synthesis, under which the learner generates new queries with only limited control over the content of the query. Indirect Active Learning can also be related to Pool-Based Sampling, e.g., if $g$ maps $x$ to its nearest neighbor in a fixed data pool $S \subseteq \X$; however, such a formulation violates the H\"older smoothness assumption we make on $g$, and we do not pursue this here.

While we are not aware of prior research within the Indirect Active Learning paradigm, several closely related paradigms have been studied previously:

\paragraph{Active Nonparametric Regression}
Within the Query Synthesis paradigm, \citet{castro2005faster} showed that, in the standard nonparametric regression setting (i.e., fixing $g(z) = z$), considering globally H\"older-smooth $f$, active learning provides no asymptotic advantages over passive learning. This contrasts with our results under Indirect Active Learning, where we show that active learning often asymptotically improves over passive learning (by leveraging an estimate of $g$ to better select informative queries).

\paragraph{Active Learning with Feature Noise}
Our generative model (Figure~\ref{fig:indirect_learning_model}) generalizes the Berkson errors-in-variables regression model~\citep{berkson1950there}, the special case when $\X = \Z$ and $g$ is known to be the identity function.
The Berkson model, which has been studied extensively in passive supervised learning~\citep{buonaccorsi2010measurement,carroll2006measurement}, was used by \citet{ramdas2014analysis} to study the effects of feature noise in active binary nonparametric classification. To the best of our knowledge, this is the only prior work generalizing Query Synthesis such that the learner cannot exactly select the covariate.

Note that the Berkson model is distinct from the classical errors-in-variables regression model~\citep{fuller2009measurement}, in which noise would be applied \emph{after} observing $X$ (i.e., we would have $X = g(Z)$ and $Y = f(X + \sigma_X \epsilon_X) + \sigma_Y \epsilon_Y$); the Berkson model is appropriate in our setting where noise affects the learner's ability to \emph{control} $X$ rather than to \emph{measure} $X$. One could additionally consider noise in the measurement of $X$, which has the effect of smoothing the observed relationship between $X$ and $Y$; however, this additional bias often makes the learning problem statistically intractable, with optimal convergence rates decaying only polylogarithmically with the sample size $n$ under common (e.g., Gaussian) noise models~\citep{fan1993nonparametric}. Measurement error in $X$ could potentially be mitigated using knowledge of $Z$~\citep[e.g., using two-stage regression, as in instrumental variable methods;][]{sargan1958estimation}, which we leave as a direction for future work.

\paragraph{Active Learning on Manifolds}
When the dimension $d_Z$ of $Z$ is smaller than that of $X$ and $g$ is sufficiently smooth, the model $X = g(Z) + \sigma_X \epsilon_X$ implies that $X$ lies near a low-dimensional manifold in $\X$. This is a natural model under which to study active learning on an unknown manifold (i.e., in the presence of unknown structural constraints on the covariate $\X$). To the best of our knowledge, no statistical theory has been developed for this problem, although a number of practical methods have been proposed~\citep{zhou2014active,li2020active,huang2020general,sreenivasaiah2021meal}.
Our results show that, whereas rates for passive learning often depend on the total intrinsic dimension $d_X + d_Z/s_g$ of the distribution of $X$, rates for active learning depend mostly on the smaller dimension $d_X$ of the noise $\epsilon_X$ in $X$. Note that both of these intrinsic dimensions may be much smaller than the extrinsic dimension of $X$, which does not appear in our rates.


\paragraph{Approximate Bayesian Computation}

Approximate Bayesian Computation (ABC) is an approach to estimating a posterior distribution for parameters of a complex model that avoids explicit computation of a likelihood function by using a simulation to relate the parameter to measurable data~\citep{csillery2010approximate,sunnaaker2013approximate}. The two-stage Algorithm~\ref{alg:active_learning_algorithm} we propose is similar in spirit to the rejection algorithm typically used in ABC~\citep{csillery2010approximate}, in that different values of the input parameter $Z$ are first sampled randomly and then evaluated in terms of similarity between experimental and target measurements $X$. However, rather than to construct a posterior distribution for the ``true value'' of the input parameters $Z$ of a simulation, our goal is to estimate an outcome $Y$ of the simulation. Moreover, our formulation and analysis of Indirect Active Learning is completely frequentist, although, as noted previously, the distribution $P_Z$ can be regarded as a prior on $z_0$.

\paragraph{Instrumental Variables Methods}
Our model of Indirect Active Learning (Figure~\ref{fig:indirect_learning_model}) assumes that the relationship between $X$ and $Y$ is unconfounded given $Z$, i.e., that $\epsilon_Y$ is independent of $X$ and that $Y$ is conditionally independent of $\epsilon_X$ given $X$. In general, as illustrated in Figure~\ref{fig:instrumental_variable_model}, there may exist unobserved confounders that affect both $X$ and $Y$, leading to the causal model assumed by instrumental variable methods, which attempt to use an unconfounded \emph{instrumental variable} $Z$ to deconfound the relationship between $X$ and $Y$, typically by projecting $X$ onto its variation caused by $Z$~\citep{baiocchi2014instrumental}.
\begin{figure}[hptb]
    \centering
    \begin{tikzpicture}
        \node[shape=circle,draw=black] (Z) at (0,0) {$Z$};
        \node[shape=circle,draw=black] (E1) at (2,-1.5) {$\epsilon$};
        \node[shape=circle,draw=black] (X) at (2,0) {$X$};
        \node[shape=circle,draw=black] (Y) at (4,0) {$Y$};
        \path [->](Z) edge node[below] {$g$} (X);
        \path [->](E1) edge node[left] {} (X);
        \path [->](X) edge node[below] {$f$} (Y);
        \path [->](E1) edge node[left] {} (Y);
    \end{tikzpicture}
    \caption{Structural causal model assumed by instrumental variable methods.}
    \label{fig:instrumental_variable_model}
\end{figure}
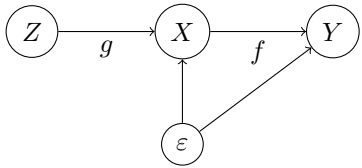
In the passive case, \citet{singh2019kernel} provided minimax rates for nonparametric instrumental variable regression, but we are not aware of work characterizing the potential benefits of actively manipulating the instrumental variable $Z$. Indeed, as we discuss further in Section~\ref{sec:conclusions}, one motivation for the present work is to provide a foundation for such work by first characterizing minimax-optimal learning in the absence of confounders.


\section{BOUNDS FOR PASSIVE LEARNING}
\label{sec:passive_case}
To provide an appropriate baseline for performance in Indirect Active Learning, we first consider the problem of \emph{passively} estimating $f(x_0)$ from an IID dataset $\{(X_i,Y_i,Z_i)\}_{i = 1}^n\}$. Upon first glance, since observations of $Z$ provide no information about the relationship between $X$ and $Y$, one might suppose that this passive problem reduces to standard passive regression, in which we have only IID observations of $(X,Y)$. However, while observations of $Z$ are indeed uninformative here, the assumption $X = g(Z) + \sigma_X \epsilon_X$ imposes constraints on the distribution of $X$ that differ from the standard passive regression setting, resulting in significantly different convergence rates. Obtaining minimax rates for the passive case will therefore facilitate interpretation of our main results for the active case, especially to understand the benefit of active learning over passive learning.

In the passive case, we consider a $k$-nearest neighbor ($k$-NN) regression estimate of $f(x_0)$:
\begin{definition}[Nearest-Neighbor Order and $k$-Nearest Neighbor Regressor]
    Given a point $x_0$ lying in a metric space $(\X, \rho)$ and a finite dataset $X_1,...,X_n \in \X$, the \emph{nearest-neighbor order $\pi(x_0;\{X_1,...,X_n\})$} is a permutation of $[n]$ such that
    \begin{align*}
        \rho \left( x_0, X_{\pi_1(x_0;\{X_1,...,X_n\})} \right)
        & \leq \rho \left( x_0, X_{\pi_2(x_0;\{X_1,...,X_n\})} \right) \\
        & \leq \cdots \\
        & \leq \rho \left( x_0, X_{\pi_n(x_0;\{X_1,...,X_n\})} \right),
    \end{align*}
      with ties broken arbitrarily.
    The $k$-nearest neighbor regressor $\hat f_{x_0,k}$ of $f$ at $x_0$ is defined by
    \begin{equation}
        \hat f_{x_0,k} := \frac{1}{k} \sum_{i = 1}^k Y_{\pi_i(x_0; \{X_1,...,X_n\})},
        \label{eq:knn_regressor}
    \end{equation}
    \label{def:nearest_neighbors}
\end{definition}
Crucially, in contrast to most other approaches to nonparametric regression, $k$-NN regression automatically adapts to unknown low-dimensional structure in the distribution of $X_1,...,X_n$~\citep{kpotufe2011k}, which is important in our setting because the model $X = g(Z) + \sigma_X \epsilon_X$ can constrain the intrinsic dimension of $X$. In contrast, we expect many methods, such as kernel methods with uniform bandwidth, to be sub-optimal in our setting.
The following result, proven in Appendix~\ref{app:passive_upper_bound}, bounds the error of the $k$-NN estimate~\eqref{eq:knn_regressor}:
\begin{theorem}[Passive Upper Bound]
    Let $\delta \in (0, 1)$ and let $\hat f_{x_0,k}$ denote the $k$-NN regression from Eq.~\eqref{eq:knn_regressor} with $k$ chosen according to Eq.~\eqref{eq:passive_k} in Appendix~\ref{app:passive_upper_bound}.
    Under Assumptions A1)-A5), there exists a constant $C > 0$, depending only on $L_g$, $s_g$, $L_f$, $s_f$, $d_X$, $d_Z$, such that
    \[\Pr \left[ \left| \hat f_{x_0,k} - f(x_0) \right|
        > C \Phi_P \log \sfrac{1}{\delta} \right] \leq \delta,\]
    where
    \begin{align}
        \Phi_P
        & = \max \Bigg\{
            \notag
            \left( \frac{1}{n} \right)^{\frac{s_fs_g}{d_Z}},
            \left( \frac{\sigma_X^{d_X}}{n} \right)^{\frac{s_f}{d_X + d_Z/s_g}}, \\
            & \left( \frac{\sigma_Y^2}{n} \right)^{\frac{s_f}{2s_f + d_Z/s_g}},
            \left( \frac{\sigma_X^{d_X} \sigma_Y^2}{n} \right)^{\frac{s_f}{2s_f + d_X + d_Z/s_g}}
        \Bigg\}.
      \label{ineq:passive_upper_bound}
    \end{align}
    \label{thm:passive_upper_bound}
\end{theorem}

Before discussing this result further, we present a matching lower bound, proven in Appendix~\ref{app:passive_lower_bounds}, showing that the rate $\Phi_P$ in Theorem~\ref{thm:passive_upper_bound} is the optimal in the passive case:

\begin{theorem}[Passive Lower Bound]
    There exist constants $c_1, c_2, n_0 > 0$, depending only on $L_g$, $s_g$, $L_f$, $s_f$, $d_X$, $d_Z$, such that, for all $n \geq n_0$, for any passive estimator $\hat f_{x_0}$,
    \begin{align}
        \sup_{g,f,P_Z,\epsilon_X,\epsilon_Y} \Pr_D \Bigg[ \left| \hat f_{x_0} - f(x_0) \right| \geq c_1 \Phi_P \Bigg] \geq c_2,
        \label{ineq:passive_lower_bound}
    \end{align}
    \label{thm:passive_lower_bound}
    where $\Phi_P$ is the passive minimax rate defined in Eq.~\eqref{ineq:passive_upper_bound} and the supremum is taken over all $g$, $f$, $P_Z$, $\epsilon_X$, $\epsilon_Y$ satisfying Assumptions A1)-A5).
\end{theorem}

Comparing Theorems~\ref{thm:passive_upper_bound} and \ref{thm:passive_lower_bound}, we see that the $k$-NN regressor in Eq.~\eqref{eq:knn_regressor} achieves the minimax optimal rate $\Phi_P$, given in Eq.~\eqref{ineq:passive_upper_bound}, for the passive case.
Each of the terms in $\Phi_P$ can dominate, depending on the relative magnitudes of $\sigma_X$, $\sigma_Y$, and $n$. If the covariate noise level $\sigma_X$ is sufficiently small, one can show that the distribution of $X$ concentrates near a manifold of dimension $d_Z/s_g$ within $\X$. Hence, depending on the response noise level $\sigma_Y$, we obtain the rate $\asymp n^{-\frac{s_fs_g}{d_Z}}$ or $\asymp n^{-\frac{s_f}{2s_f + d_Z/s_g}}$, corresponding to the rate of nonparametric regression over a covariate space of dimension $d_Z/s_g$, in the absence of response noise~\citep{kohler2014optimal} or the presence of response noise~\citep{tsybakov2008introduction}, respectively.
Increasing the level $\sigma_X$ of covariate noise leads to slower rates, respectively
$\asymp n^{-\frac{s_f}{d_X + d_Z/s_g}}$ or $\asymp n^{-\frac{s_f}{2s_f + d_X + d_Z/s_g}}$, equivalent to increasing the dimension of the covariate space by $d_X$.
An alternative interpretation is as a multi-resolution bound: fixing $\sigma_X > 0$ and $\sigma_Y > 0$, the convergence rate slows as $n \to \infty$, initially from $n^{-s_fs_g/d_Z}$ to $n^{-\frac{s_f}{d_X + d_Z/s_g}}$ (if $d_X \geq 2s_f$) or $n^{-\frac{s_f}{2s_f + d_Z/s_g}}$ (if $d_X \leq 2s_f$), and then further to $n^{-\frac{s_f}{2s_f + d_X + d_Z/s_g}}$.

\section{BOUNDS FOR ACTIVE LEARNING}
\label{sec:active_case}
We now turn to the active case. Our upper bounds use a ``two-stage'' algorithm, detailed in  Algorithm~\ref{alg:active_learning_algorithm}, which sequentially performs two passive experiments, a ``pure exploration'' experiment followed by a ``pure exploitation'' experiment. In the first stage, the algorithm randomly samples the control variable $Z$ according to the given ``prior'' distribution $P_Z$ (as in the passive algorithm). Each sample value of $Z$ is repeated several times, and the resulting values of $X$ are averaged to determine which observed value of $Z$ is most likely to return a value of $X$ near $x_0$. The algorithm then repeatedly samples this value of $Z$ (hopefully obtaining a large number of samples with $X$ near $x_0$), and finally applies the $k$-NN regressor from Eq.~\eqref{eq:knn_regressor} over the resulting dataset.
\begin{algorithm2e}[htb]
    \DontPrintSemicolon
      \KwInput{Budget $n$,
               integers $\ell, k \in [1, n/2]$,
               distribution $P_Z$ on $\Z$,
               target point $x_0 \in \X$}
      \KwOutput{Estimate $\hat f_{x_0}$ of $x_0$}
      $m \leftarrow \lfloor n/2 \rfloor$ \\
      \For(\tcp*[h]{exploration stage}){$i \in [\lfloor m/\ell \rfloor]$}{
        Draw $Z_i \sim P_Z$ \\
        \For{$l \in [\ell]$}{
            Observe $X_{i,l} = g(Z_i) + \sigma_X \epsilon_{X,i,l}$ and $Y_{i,l} = f(X_{i,l}) + \sigma_Y \epsilon_{Y,i,l}$ \\
        }
        Compute $\bar X_i = \sum_{l = 1}^\ell X_{i,l}$ \\
      }
      $i^* \leftarrow \argmin_{i \in [\lfloor m/\ell \rfloor]} \left\| \bar X_i - x_0 \right\|$ \\
      \For(\tcp*[h]{exploitation stage}){$i \in [n - m]$}{
        Sample $X_{m+i} = g(Z_{i^*}) + \sigma_X \epsilon_{X,m+i}$ and $Y_{m+i} = f(X_{m+i}) + \sigma_Y \epsilon_{Y,m+i}$ \\
      }
      \Return{$\hat f_{x_0} = \frac{1}{k} \sum_{j = 1}^k Y_{\pi_j(x_0; \{X_1,...,X_n\})}$}
    \caption{Two-stage active learning algorithm.}
    \label{alg:active_learning_algorithm}
\end{algorithm2e}

Because passive experiments are often much easier to carry out (and parallelize) than sequential active experiments, this two-stage design may be practical in many settings where a fully iterative procedure would be expensive or time-consuming to implement. Nevertheless, Algorithm~\ref{alg:active_learning_algorithm} achieves the optimal convergence rate in most, though not all, experimental settings, and typically converges faster than the passive minimax rate. In particular, we have the following upper bound, proven in Appendix~\ref{app:active_upper_bound}:
\begin{theorem}[Active Upper Bound]
    Let $\delta \in (0, 1)$, and let $\hat f_{x_0}$ denote the estimator proposed in Algorithm~\ref{alg:active_learning_algorithm}, with $k$ and $\ell$ chosen according to Eqs.~\eqref{eq:active_k} and \eqref{eq:active_ell} in Appendix~\ref{app:active_upper_bound}.
    Then, under Assumptions A1)-A5), for some constant $C > 0$ depending only on $L_g$, $s_g$, $L_f$, $s_f$, $d_X$, and $d_Z$, with probability $\geq 1 - \delta$, $\left| \hat f_{x_0} - f(x_0) \right| \leq C \Phi_A^* \log \sfrac{1}{\delta}$, where
    \begin{align}
        \Phi_A^*
        := & \max \Bigg\{
            \notag
            \left( \frac{1}{n} \right)^{\frac{s_fs_g}{d_Z}},
            \left( \frac{\sigma_X^{d_X}}{n} \right)^{\frac{s_f}{d_X}},
            \frac{\sigma_Y}{\sqrt{n}}, \\
            \label{rate:active_upper_bound}
            & \left( \frac{\sigma_Y^2 \sigma_X^{d_X}}{n} \right)^{\frac{s_f}{2s_f + d_X}},
            \left( \frac{\sigma_X^2 \log n}{n} \right)^{\frac{s_fs_g}{2s_g + d_Z}}
        \Bigg\}.
    \end{align}
    \label{thm:active_upper_bound}
\end{theorem}
Meanwhile, we have the following minimax lower bound, proven in Appendix~\ref{app:active_lower_bounds}, on the error of any active learner:
\begin{theorem}[Lower Bounds for the Active Case]
    There exist constants $c_1, c_2 > 0$, depending only on $L_g$, $s_g$, $L_f$, $s_f$, $d_X$, and $d_Z$, such that, for any active estimator $\hat f_{x_0}$
    \vspace{-2mm}
    \[\sup_{g,f,P_Z,\epsilon_X,\epsilon_Y}
            \Pr_{\{(Z_i,X_i,Y_i)\}_{i = 1}^n} \left[ \left| \hat f_{x_0} - f(x_0) \right| \geq c_1 \Phi_{A,*} \right]
        \geq c_2,
    \vspace{-2mm}\]
    where the supremum is over all $g$, $f$, $P_Z$, $\epsilon_X$, and $\epsilon_Y$ satisfying Assumptions A1)-A5), and
    \begin{align}
        \Phi_{A,*}
        = \max \Bigg\{
            \notag
            & \left( \frac{1}{n} \right)^{\frac{s_fs_g}{d_Z}},
            \left( \frac{\sigma_X^{d_X}}{n} \right)^{s_f/d_X}, \\
            & \frac{\sigma_Y}{\sqrt{n}},
            \left( \frac{\sigma_X^{d_X} \sigma_Y^2}{n} \right)^{\frac{s_f}{2s_f + d_X}}
        \Bigg\}.
        \label{rate:active_lower_bound}
    \end{align}
    \label{thm:active_lower_bound}
\end{theorem}

Comparing the upper bound $\Phi_A^*$ (Eq.~\eqref{rate:active_upper_bound}) and lower bound $\Phi_{A,*}$ (Eq.~\eqref{rate:active_lower_bound}), we can write
\[\Phi_A^* = \max \left \{
    \Phi_A^*,
    \left( \frac{\sigma_X^2 \log n}{n} \right)^{\frac{s_fs_g}{2s_g + d_Z}}
\right\};\]
i.e., the first four terms in $\Phi_A^*$ match the corresponding terms in $\Phi_{A,*}$, while the fifth term in $\Phi_A^*$ may be sub-optimal. Hence, Algorithm~\ref{alg:active_learning_algorithm} is optimal whenever $\sigma_Y$ is sufficiently large or both $d_X$ and $\sigma_X$ are sufficiently small; specifically, for $\sigma_Y$ small, the fifth term $\left( \frac{\sigma_X^2 \log n}{n} \right)^{\frac{s_fs_g}{2s_g + d_Z}}$ dominates $\Phi_A^*$ only when
$1
    \lesssim_{\log} \sigma_X n^{\frac{s_g}{d_Z}}
    \lesssim_{\log} n^{\frac{2s_g + d_Z}{d_X d_Z}}$.



\section{DISCUSSION}
\label{sec:discussion}

\paragraph{Comparison of Active and Passive Bounds}
The first terms of the passive minimax rate $\Phi_P$ (Eq.~\eqref{ineq:passive_upper_bound}) and our upper bound rate $\Phi_A^*$ (Eq.~\eqref{rate:active_upper_bound}) match, while, $\Phi_A^*$ strictly improves on the second, third, and fourth terms of $\Phi_P$ by removing $d_Z/s_g$ from the exponents. Meanwhile, the last term of $\Phi_A^*$ dominates $\Phi_P$ only if $\sigma_Y$ is sufficiently small and both $d_X \leq 2$ and $\sigma_X \gtrsim_{\log} n^{-d_Z/s_g}$. In particular, whenever $\sigma_Y$ is sufficiently large or both $d_X > 2$ and $\sigma_X \gtrsim_{\log} n^{-d_Z/s_g}$, (i.e., whenever the response or covariate noise is sufficiently large), Algorithm~\ref{alg:active_learning_algorithm} strictly outperforms the best passive algorithm.
Intuitively, the advantage of active learning comes from the ability to learn the structure of $g$ and thereby remove the degrees of freedom in $X$ that come from variation in $Z$.

\paragraph{Two-Stage vs. Fully Active Algorithms}
Our proposed two-stage Algorithm~\ref{alg:active_learning_algorithm} utilizes active learning in only a very limited way, as prior information from preceding iterations is used only once, in between exploration and exploitation stages, rather than in each iteration.
Practically, this demonstrates that running a pilot experiment to optimize free parameters of a main experiment can provide a dramatic improvement in the utility of data from the main experiment. That is, Algorithm~\ref{alg:active_learning_algorithm} may be useful not only in the specific setting of Active Learning, but also in the more general setting of Experimental Design, where one wishes to design an optimal experiment before collecting (most of) the data. Moreover, when possible, data collection within each of the two stages can be trivially parallelized. On the other hand, sequentially optimizing these parameters separately before collecting each sample may be prohibitively complex, laborious, or time-consuming.

One may speculate that this simplicity causes the gap between our upper and lower bounds for the active case (Theorems~\ref{thm:active_upper_bound} and \ref{thm:active_lower_bound}), and that a more sophisticated algorithm might iteratively utilize all of the past information to inform each future decision.
On the other hand, the evaluation function, namely the error of the final estimate, is more comparable to the notion of Simple Regret, rather than Cumulative Regret, used to evaluate many active algorithms~\citep{slivkins2019introduction}. For example, in our problem, an algorithm can afford to perform $\Omega(n)$ steps of pure exploration (as in Algorithm~\ref{alg:active_learning_algorithm}), whereas any such algorithm necessarily exhibits linear cumulative regret.
Since, in many problems (e.g., multi-armed bandits), simple two-stage algorithms can be optimal under Simple Regret, it remains unclear whether the gap between Theorems~\ref{thm:active_upper_bound} and \ref{thm:active_lower_bound} is inherent to two-stage algorithms such as Algorithm~\ref{alg:active_learning_algorithm} or is an artifact of our (upper or lower) bounds.

\section{EXPERIMENTAL RESULTS}
\label{sec:experimental_results}

In this section, we present the results of two numerical experiments designed to numerically validate our theoretical findings. The main goal was to investigate whether the theoretical performance advantages of the two-stage algorithm (over the Passive algorithm) predicted by our theoretical results translate into apparent practical advantages, and to shed light on how these performance advantages scale with various parameters of the data-generating process. Python code and instructions for replicating our results are available at \url{https://gitlab.tuebingen.mpg.de/shashank/indirect-active-learning}.

\subsection{Experiment 1: Synthetic Data}
\label{subsec:synthetic_data}

We first consider a setting with simple synthetic data where we can directly manipulate features of the data-generating process, such as the dimensions $d_Z$ of $P_Z$ and $d_X$ of $\epsilon_X$, the smoothness $s_g$ of $g$, and the noise levels $\sigma_X$ and $\sigma_Y$.

\paragraph{Algorithms}

As our Passive algorithm, we utilized standard $k$-NN regression, as in Theorem~\ref{thm:passive_upper_bound}. As our Active algorithm, we used the two-stage algorithm presented in Algorithm~\ref{alg:active_learning_algorithm}.
For comparison, we also evaluated an ``Oracle'' estimator that always samples $(X, Y, Z)$ conditional on $Z = z_0$ and then runs a $k$-NN regression of $Y$ over $X$. Since this relies on knowledge of $z_0$, this is not feasible in practice, but comparing the performance of the Active estimator to this ideal estimator will allow us to discern how much of the Active estimator's error comes from learning $f$ versus learning $g$.

Although Theorems~\ref{thm:passive_upper_bound} and \ref{thm:active_upper_bound} specify optimal values of the hyperparmeters $k$ and $\ell$ of the active and passive estimators, respectively, these optimal values depend on parameters, such as the local smoothness of $g$ and $f$ around $z_0$ and $x_0$, that are typically unknown in practice. We therefore considered two versions of each algorithm: a version utilizing the theoretically optimal values of each parameter (according to Eqs.~\eqref{eq:passive_k}, \eqref{eq:active_k}, and \eqref{eq:active_ell}) and a fully data-dependent version that used $4$-fold cross-validation to select each hyperparameter.

Hence, in total, we compared $6$ algorithms, namely Passive, Active, and Oracle algorithms with theoretically optimal hyperparameter values, as well as their cross-validated counterparts Passive CV, Active CV, and Oracle CV.

\paragraph{Data Generating Process} Synthetic data was generated as follows. $P_Z$ was uniform on $[0, 1]^{d_Z}$, $\epsilon_X \sim \mathcal{N}(0, I_{d_X})$, $\epsilon_Y \sim \mathcal{N}(0, 1)$, $g : \R^{d_Z} \to \R^{d_X}$ defined by
$g_1(z) = \cdots = g_{d_X}(z) = \sum_{j = 1}^{d_Z} |z_j|^{s_g}$
for all $z \in \R^{d_Z}$, and $f : \R^{d_X} \to \R$ defined by $f(x) = \|x\|_2$ for all $x \in \R^{d_X}$.
One can then check that $g \in \C^{s_g}(\Z, z_0; \Z; 1)$ and $f \in \C^1(\X, x_0; \Y; 1)$.

In the default setup, we set $n = 2^{10}$, $s_g = 1$, $d_Z = d_X = 3$, and $\sigma_X = \sigma_Y = 1$. We then varied each of these parameters to demonstrate its effect on the performance of each estimator. Specifically, we varied:
\begin{itemize}[noitemsep,topsep=0pt]
    \item $n \in \{2^4, 2^5, ..., 2^{12}\}$
    \item $s_g \in \{0.1, 0.2, ..., 1.0\}$
    \item $d_Z, d_X \in \{1, 2, ..., 10\}$
    \item $\sigma_X, \sigma_Y \in \{10^{-1}, 10^{-0.8}, 10^{-0.6}, ..., 10^{0.8}, 10^1\}$
\end{itemize}

\paragraph{Results}
Each experiment was replicated independently $2^{10}$ times. Figure~\ref{fig:experimental_results} displays means and standard errors (across IID replicates) of the mean squared error of each algorithm.
\begin{figure*}[htbp]
    \centering
    \includegraphics[width=0.95\linewidth,clip,trim={0.8cm 0.2cm 0.8cm 1.1cm}]{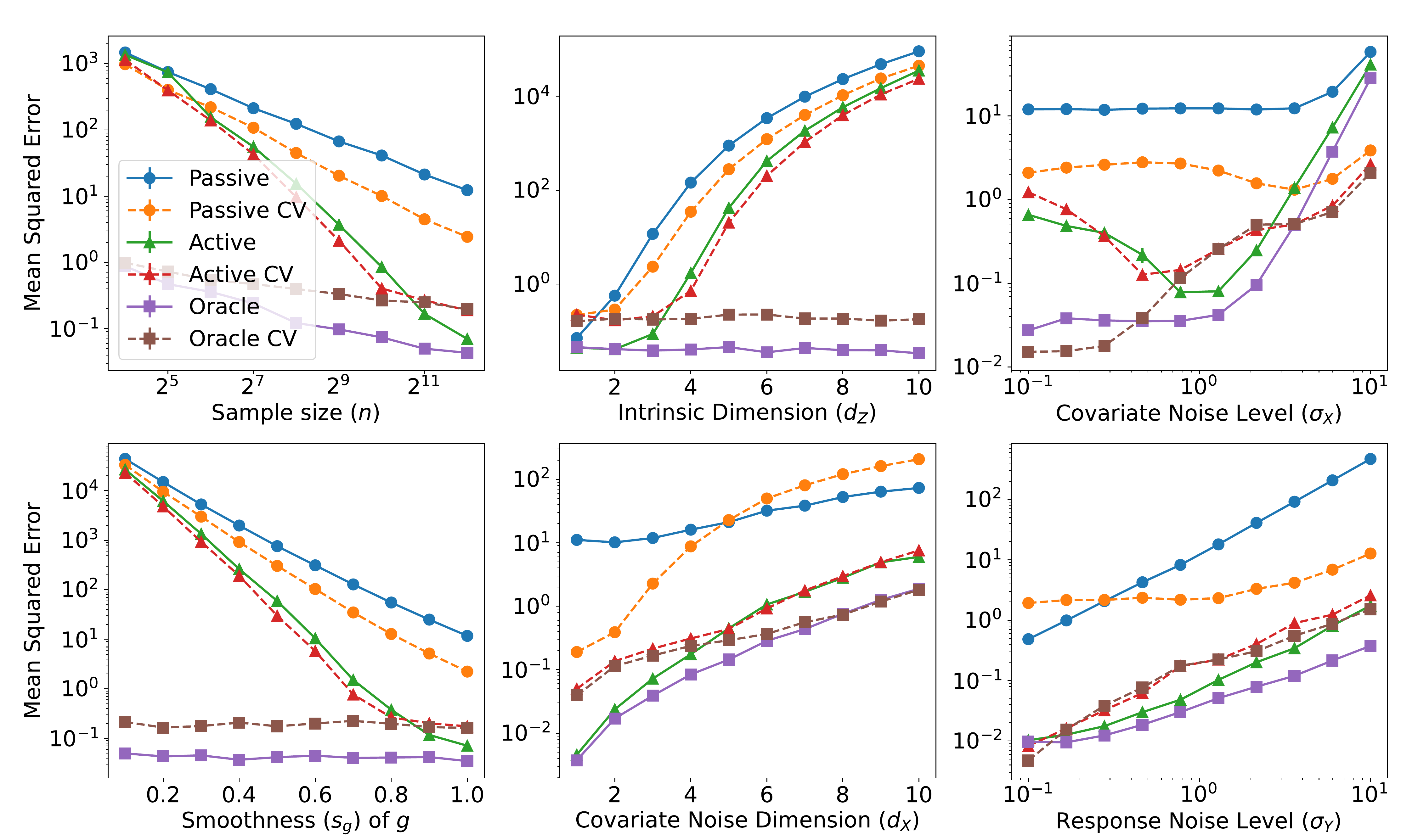}
    \caption{Experiment 1: Mean (over $2^{10}$ IID replicates) squared error of each estimator, over parameters of the data-generating process. Error bars denoting standard errors over replicates are too small to be visible. Note logarithmic $y$-axes.}
    \label{fig:experimental_results}
\end{figure*}
Unsurprisingly, Passive estimators consistently exhibit the largest error, with Passive CV typically outperforming Passive, and the Oracle estimators consistently exhibit the smallest error, with Oracle typically slightly outperforming Oracle CV. Meanwhile, performance of the Active estimators lies between those of the corresponding Passive and Oracle estimators.

Whether Active estimators perform more like corresponding Passive estimators or corresponding Oracle estimators depends primarily on how easy it is to learn the function $g$ and thereby select values of $Z \approx z_0$. Hence, the advantage of Active estimators over their Passive counterparts increases with sample size $n$ and smoothness $s_g$, decreases with intrinsic dimension $d_Z$ and covariate noise level $\sigma_X$, and is essentially unaffected by covariate noise dimension $d_X$ and response noise level $\sigma_Y$.

The performances of the Active and Active CV estimators are generally quite similar. This is encouraging, as Active CV, though computationally more demanding, is more practical in most real applications where the smoothness levels of $g$ and $f$ are unknown. Since the Active estimator is minimax optimal in most settings, we may similarly expect Active CV to be minimax optimal in most settings.






\subsection{Experiment 2: Epidemiological Forecasting}
\label{subsec:SIRD}

We now consider a more realistic data-generating process based on
a stochastic SIRD model, a system of discrete-time stochastic differential equations (SDEs) that forms the basis for many models of infectious disease spread used in epidemiology~\citep{bailey1975mathematical,arik2020interpretable}. Given a fixed population of individuals the SIRD model tracks the number of individuals in each of four states: individuals (S)usceptible to contracting the disease, (I)nfected individuals who carry the disease and may infect others, (R)ecovered individuals who are assumed to be immune to the disease, and (D)eceased individuals who died from the disease. Given SIRD counts at a point in time, a system of SDEs (given in Appendix~\ref{app:SIRD}) determines how SIRD compartment counts change at the next time step, allowing one to simulate SIRD compartment counts over time.

\paragraph{Problem Setup}
We consider the common epidemiological challenge of predicting the outcome of an epidemic disease outbreak, using only data from the beginning of the outbreak~\citep{brooks2015flexible,desai2019real,arik2020interpretable}. Specifically, in a population of $N = 1000$ individuals, given SIRD counts from the first $20$ of a simulation ($X = (S_{1:20}, I_{1:20}, R_{1:20}, D_{1:20}) \in \mathbb{N}^{80}$), we aim to predict the total number $Y = D_{100} \in \mathbb{N}$ of deaths by day $T = 100$.

Since key parameters $Z = (\beta, \gamma, \delta) \in [0, 1]^3$ determining the rates at which individuals transition between SIRD compartments are unknown, one cannot simply propagate SDEs starting from the given measurement data. One can, however, use this simulation to generate a dataset of $(X, Y, Z)$ triples for various values of $Z$ and then train a machine learning model to predict $Y$ from $X$. This can be compared to the approach of \citet{brooks2015flexible}, who forecast current seasonal influenza epidemics based on outcomes from historical epidemic seasons exhibiting greatest similarity to data collected so far in the current season. Since we have access to a simulation, we can try to improve forecasts by selecting values of $Z$ intelligently, rather than completely randomly. This is precisely what Indirect Active Learning attempts to do (in a similar spirit as Approximate Bayesian Computation).





\paragraph{Results}
We evaluated the performance of both our two-stage Active learning Algorithm~\ref{alg:active_learning_algorithm} and the Passive $k$-NN regressor on this task, for various sample sizes (i.e., simulation budgets) $n$.
Since we do not know the true smoothness levels $s_g$ and $s_f$ in this scenario, we selected estimator hyperparameters using $4$-fold cross-validation, as in the previous experiment.
\begin{figure}
    \centering
    \includegraphics[width=0.95\linewidth,clip,trim={0cm 0cm 3cm 2cm}]{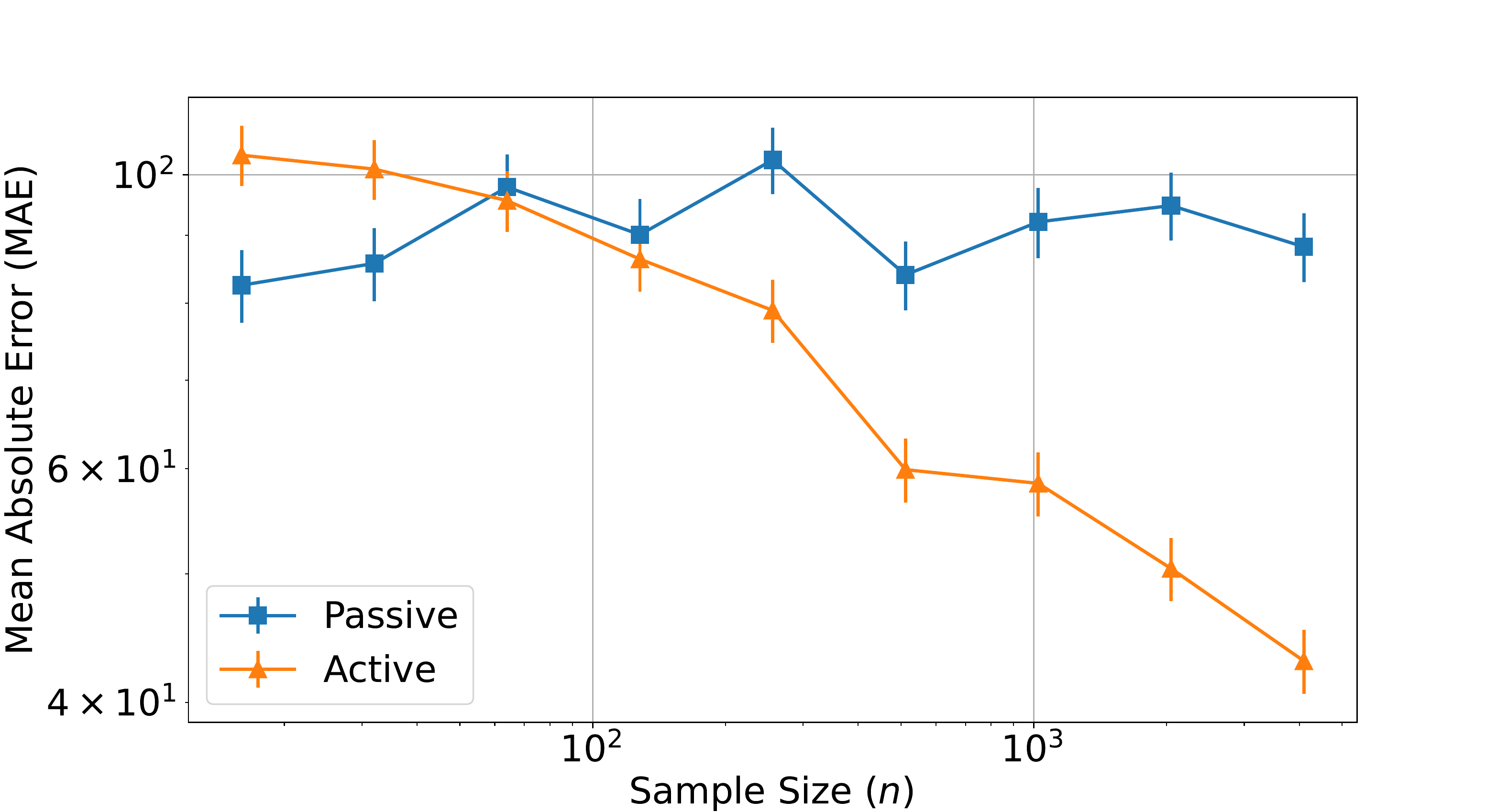}
    \caption{Experiment 2: Mean absolute error of Active and Passive estimates of $D_{100}$, as a function of sample size $n$.}
    \label{fig:SIR_performance}
\end{figure}
As shown in Figure~\ref{fig:SIR_performance}, as sample size $n$ increases from $2^4$ to $2^{14}$, mean absolute error of the Active estimate decreases. Meanwhile, the error of the Passive estimator fails to improve perceptibly.
Figure~\ref{fig:infected_plots} visualizes why the Active algorithm outperforms the Passive algorithm, namely that it preferentially samples simulated scenarios that are similar to, and therefore useful for predicting, the test scenario.
\begin{figure}
    \centering
    \includegraphics[width=0.95\linewidth]{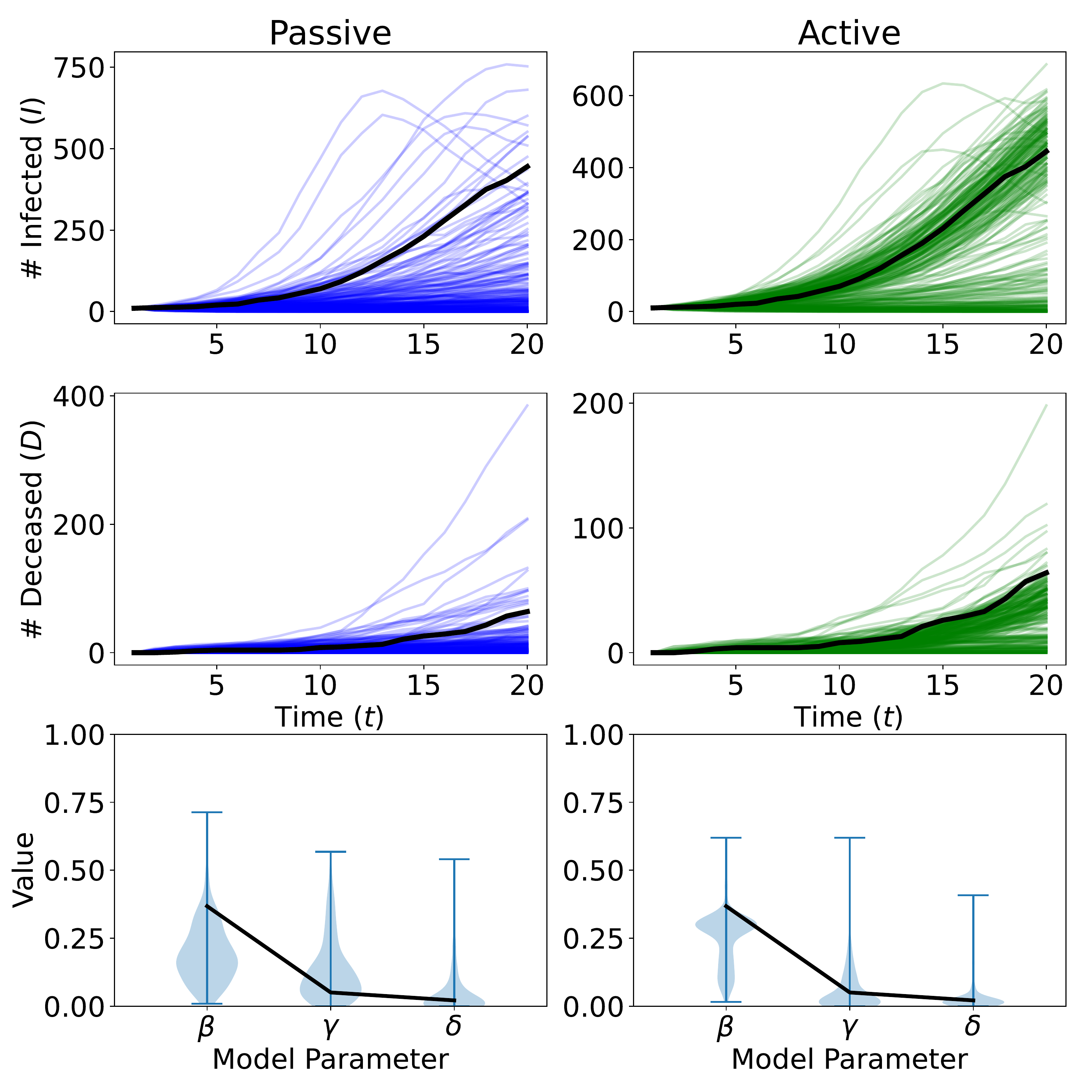}
    \caption{Experiment 2: Example distributions of Infected ($I_{1:20}$) and Deceased ($D_{1:20}$) trajectories over the first $20$ days of the outbreak, as well as underlying $Z = (\beta, \gamma, \delta)$ parameters, sampled according to the Passive and Active algorithms. Black lines indicate the test scenario $(x_0, z_0)$ being queried.}
    \label{fig:infected_plots}
\end{figure}

\section{CONCLUSIONS \& FUTURE WORK}
\label{sec:conclusions}

This paper studied the nonparametric minimax theory of Indirect Active Learning and a corresponding passive baseline. We showed that active learning asymptotically outperforms passive learning in many cases, and that most of this improvement is realized by a two-stage algorithm using only a very limited form of active learning, which may be easier to implement than fully active algorithms. Finally, we illustrated how how behaviors of these algorithms change with parameters of a synthetic data-generating process and demonstrated benefits of two-stage active learning in an epidemiological forecasting application.

Besides closing the gap between our active upper and lower bounds, several interesting directions remain for future work. First, we focused here on a local nonparametric problem of estimating $f(x_0)$ at a single point $x_0$. In many cases, it may be more useful or tractable to estimate a global parametric relationship, such as a linear relationship $f(x) = \beta^\intercal x$. One could additionally assume that $g$ is linear or that either or both of $f$ and $g$ depend sparsely on their inputs, allowing for results that scale better with dimensions $d_Z$ and $d_X$ than our nonparametric results.

Second, as discussed in Section~\ref{sec:related_work}, our model of indirect learning assumes the relationship between $X$ and $Y$ is unconfounded given $Z$, while introducing confounding leads to the instrumental variable setting.
When $Z$ is actively manipulated, the relationship between $Z$ and $Y$ is automatically likely to be unconfounded, satisfying a key assumptions of instrumental variable methods.
Furthermore, we conjecture that active learning approaches such as ours could help address the \emph{problem of weak instruments}~\citep{stock2002survey,murray2006avoiding,olea2013robust} by selecting values of $Z$ that induce maximal variation in $X$.


\subsubsection*{Acknowledgements}

The author would like to thank Jia-Jie Zhu, Krikamol Maundet, and Bernhard Sch\"olkopf for helpful discussions. This work was supported by the German Federal Ministry of Education and Research (BMBF) through the T\"ubingen AI Center (FKZ: 01IS18039B).

\bibliographystyle{plainnat}
\bibliography{biblio}

\onecolumn
\appendix

\section{Upper Bound for the Passive Case (Theorem~\ref{thm:passive_upper_bound})}
\label{app:passive_upper_bound}

In this Appendix, we prove our main upper bound for the passive case. We begin by reiterating the result:

\begin{customthm}{\ref{thm:passive_upper_bound}}[Main Passive Upper Bound]
    Suppose that $Z$ has dimension $d_Z$ around some $z_0 \in \Z$ with $g(z_0) = x_0$. Suppose $\epsilon_X$ is sub-Gaussian and has dimension $d_X$ around $0$.
    Suppose that the response noise $\epsilon_Y$ is sub-Gaussian.
    Suppose that $g \in \C^{s_g}(\Z, \X)$ and $f \in \C^{s_f}(\X, \Y)$. Then, there exists a constant $C > 0$ such that, for any $\delta \in (0, 1)$, with probability at least $1 - \delta/2 - 2e^{-k/4}$,
    \[\left| \hat f_{x_0,k} - f(x_0) \right|
      \leq C \left( \sigma_Y \sqrt{\frac{\log \sfrac{1}{\delta}}{k}}
      + \max \left\{ \left( \left( \frac{k}{n} \right)^{\frac{s_g}{d_Z}} \sqrt{\log \frac{n}{\delta}} \right)^{s_f}, \left( \frac{k \sigma_X^{d_X}}{n} \right)^{\frac{s_f}{d_X + d_Z/s_g}} \right\} \right).\]
    In particular, setting
    \begin{equation}
        k = \max \left\{
            4 \log \frac{4}{\delta},
            \min \left\{
                \sigma_Y^{\frac{2(d_X + d_Z/s_g)}{2s_f + d_X + d_Z/s_g}} \left( \frac{n}{\sigma_X^{d_X}} \right)^{\frac{2s_f}{2s_f + d_X + d_Z/s_g}},
                \sigma_Y^{\frac{2d_Z}{2s_fs_g + d_Z}} n^{\frac{2s_fs_g}{2s_fs_g + d_Z}}
            \right\}
        \right\},
        \label{eq:passive_k}
    \end{equation}
    gives, with probability $\geq 1 - \delta$,
    $\displaystyle\left| \hat f_{x_0,k} - f(x_0) \right|
        \lesssim \Phi_P \log \frac{1}{\delta}$,
    where
    \[\Phi_P = \max \left\{
        \left( \frac{1}{n} \right)^{\frac{s_fs_g}{d_Z}},
        \left( \frac{\sigma_X^{d_X}}{n} \right)^{\frac{s_f}{d_X + d_Z/s_g}},
        \left( \frac{\sigma_Y^2}{n} \right)^{\frac{s_f}{2s_f + d_Z/s_g}},
        \left( \frac{\sigma_X^{d_X} \sigma_Y^2}{n} \right)^{\frac{s_f}{2s_f + d_X + d_Z/s_g}}
    \right\}.\]
    \label{thm:passive_upper_bound_app}
\end{customthm}

Before proving this theorem, we state and prove a few useful lemmas. Our first lemma shows that, under the Local Dimension assumption (Definition~\ref{def:local_assumption}), the distribution of $k$-NN distances is tightly concentrated. Such results are commonly used in the analysis of $k$-NN methods~\citep{kpotufe2011k,chaudhuri2014rates,biau2015lectures,singh2016finite,singh2021statistical}:

\begin{lemma}[Convergence of $k$-NN Distance]
    Let $X$ be a random variable taking values in a metric space $(\X, \rho)$, and let $X_1,...,X_n$ be IID observations of $X$. Suppose $X$ has local dimension $d$ around a point $x \in \X$. Then there exist constants $C, \kappa > 0$ such that, whenever $k/n \leq \kappa$, with probability at least $1 - e^{-k/4}$, the distance
    $\rho(x, X_{\pi_k(x; \{X_1,...,X_n\})})$ from $x$ to its $k$-nearest neighbor in $\{X_1,...,X_n\}$ is at most
    \[\rho(x, X_{\pi_k(x; \{X_1,...,X_n\})})
      \leq C \left( \frac{k}{n} \right)^{1/d}.\]
    \label{lemma:knn_distance}
\end{lemma}

\begin{proof}
    By definition of local dimension, there exist $c, r^* > 0$ such that, for all $r \in (0, r^*]$, $\prob[\rho(X, x) \leq r] \geq c r^d$. Suppose $\frac{k}{n} \leq \frac{c (r^*)^d}{2}$, so that, for $r := \left( \frac{2k}{cn} \right)^{1/d}$, $r \leq r^*$. Then,
    \[\prob[\rho(X, x) \leq r] \geq \frac{2k}{n}.\]
    Since, for each $i \in [n]$, $1\{\rho(X_i, x) \leq r\} \sim \operatorname{Bernoulli}(\prob[\rho(X, x) \leq r])$, by a multiplicative Chernoff bound, with probability at least $1 - e^{- k/4}$,
    \[\sum_{i = 1}^n 1\{\rho(X_i, x) \leq r\} \geq k,\]
    which implies $\rho(x, X_{\pi_k(x; \{X_1,...,X_n\})}) \leq r$.
\end{proof}

Our next lemma shows that local dimension is (sub-)additive:

\begin{lemma}
    Consider two independent $\R^d$-valued random variables $X_1$ and $X_2$. Under any norm $\|\cdot\| : \R^d \to \R$, suppose $X_1$ has local dimension $d_1$ around $x_1$ and $X_2$ has local dimension $d_2$ around $x_2$. Then, $X_1 + X_2$ has local dimension $d_1 + d_2$ around $x_1 + x_2$.
    \label{lemma:sum_of_dimensions}
\end{lemma}
\begin{proof}[Proof of Lemma~\ref{lemma:sum_of_dimensions}]
    By the triangle inequality and independence of $X$ and $\epsilon$,
    \begin{align*}
        \Pr[\|X_1 + X_2 - (x_1 + x_2)\| \leq r]
        & \geq \Pr[\|X_1\| \leq r/2 \text{ AND } \|X_2\| \leq r/2] \\
        & = \Pr[\|X_1\| \leq r/2] \Pr[\|X_2\| \leq r/2].
    \end{align*}
    The result follows from the definition of local dimension.
\end{proof}

Finally, the following lemma shows how local dimension of a random variable changes when applying a H\"older-smooth function:

\begin{lemma}[Dimension of Pushforward]
    Let $(\X, \rho_\X)$ and $(\Y, \rho_\Y)$ be metric spaces. Let $X$ be a random variable taking values in $\X$, and let $f \in \C^s(\X, x; \Y)$ be a map from $\X$ to $\Y$ that is locally H\"older continuous around a point $x \in \X$. If $X$ has local dimension $d$ around $x$ in $(\X, \rho_\X)$, then $f(X)$ has local dimension $d/s$ around $f(x)$ in $(\Y, \rho_\Y)$.
    \label{lemma:pushforward_dimension}
\end{lemma}

\begin{proof}
    If $\|f\|_{\C^s(\X, x; \Y)} = 0$, then $f$ is constant and the result follows trivially. Otherwise, by definition of H\"older continuity,
    \begin{align*}
        \liminf_{r \downarrow 0} \frac{\prob \left[ \rho_\Y(f(X), f(x)) \leq r \right]}{r^{d/s}}
        & = \|f\|_{\C^s(\X, x; \Y)}^{-d} \liminf_{r \downarrow 0} \frac{\prob \left[ \rho_\Y(f(X), f(x)) \leq L r^s \right]}{r^d} \\
        & \geq \|f\|_{\C^s(\X, x; \Y)}^{-d} \liminf_{r \downarrow 0} \frac{\prob \left[ \rho_\X(X, x) \leq r \right]}{r^d}
          > 0,
    \end{align*}
    since $X$ has dimension $d$ around $x$.
\end{proof}

We are now ready to prove Theorem~\ref{thm:passive_upper_bound}. The key idea of the proof is that, combining Lemmas~\ref{lemma:sum_of_dimensions} and \ref{lemma:pushforward_dimension}, $X$ has dimension $d_X + d_Z/s_g$ around $x_0$. Together with Lemma~\ref{lemma:knn_distance}, this implies that the distance between $x$ and its $k$-nearest neighbor among $\{X_1,...,X_n\}$ is of order $\left( \frac{k}{n} \right)^{\frac{1}{d_X + d_Z/s_g}}$, determining the bias due to smoothing implicit in the $k$-NN regressor. The full proof is as follows:

\begin{proof}[Proof of Theorem~\ref{thm:passive_upper_bound}]
    For any $X_1,...,X_n \in \X$, let
    \[\tilde f_{x_0,X_1,...,X_n,k} = \frac{1}{k} \sum_{i = 1}^k f(X_{\pi_i(x; \{X_1,...,X_n\})})\]
    denote the conditional expectation of $\hat f_{x_0,k}$ given $X_1,...,X_n$.
    By the triangle inequality,
    \begin{equation}
        \left| \hat f_{x_0,k} - f(x_0) \right|
          \leq \left| \hat f_{x_0,k} - \tilde f_{x_0,X_1,...,X_n,k} \right| + \left| \tilde f_{x_0,X_1,...,X_n,k} - f(x_0) \right|,
        \label{ineq:passive_bias_variance_decomposition}
    \end{equation}    
    where the first term captures stochastic error due to the response noise $\epsilon_Y$ and the second term captures smoothing bias.
    
    For the first term in Eq.~\eqref{ineq:passive_bias_variance_decomposition}, since the response noise $\{\epsilon_{Y,i}\}_{i = 1}^n$ is assumed to be additive,
    \begin{align}
        \left| \hat f_{x_0,k} - \tilde f_{x_0,X_1,...,X_n,k} \right|
        & = \left| \frac{1}{k} \sum_{i = 1}^k \sigma_Y \epsilon_{Y,\pi_i(x_0; \{X_1,...,X_n\})} \right|.
        \label{eq:passive_upper_bound_additive_noise}
    \end{align}
    Hence, noting that the response noise $\{\epsilon_{Y,i}\}_{i = 1}^n$ is independent of the covariates $\{X_i\}_{i = 1}^n$, by a standard concentration inequality for sub-Gaussian random variables (e.g., Lemma~\ref{lemma:sub_gaussian_mean_concentration} in Appendix~\ref{app:supplementary_lemmas}), with probability at least $1 - \delta/2$,
    \begin{equation}
        \left| \hat f_{x_0,k} - \tilde f_{x_0,X_1,...,X_n,k} \right|
        \leq \sigma_Y \sqrt{\frac{2}{k} \log \frac{4}{\delta}}.
        \label{ineq:passive_upper_variance_bound}
    \end{equation}
    
    For the second term in Eq.~\eqref{ineq:passive_bias_variance_decomposition}, by the triangle inequality and local H\"older smoothness of $f$ around $x_0$, for some constant $C_1 > 0$,
    \begin{align}
        \left| \tilde f_{x_0,X_1,...,X_n,k} - f(x_0) \right|
        \notag
        & = \left| \frac{1}{k} \sum_{i = 1}^k f(X_{\pi_i(x_0; \{X_1,...,X_n\})}) - f(x_0) \right| \\
        \notag
        & \leq \frac{1}{k} \sum_{i = 1}^k \left| f(X_{\pi_i(x_0; \{X_1,...,X_n\})}) - f(x_0) \right| \\
        \notag
        & \leq C_1 \frac{1}{k} \sum_{i = 1}^k \left\| x_0 - X_{\pi_i(x_0; \{X_1,...,X_n\})} \right\|^{s_f} \\
        \label{ineq:f_holder_continuity_bound}
        & \leq C_1 \left\| x_0 - X_{\pi_k(x_0; \{X_1,...,X_n\})} \right\|^{s_f}.
    \end{align}
    Hence, the remainder of this proof is devoted to bounding the distance $\left\| x_0 - X_{\pi_k(x_0; \{X_1,...,X_n\})} \right\|$ from $x_0$ to its $k$-nearest neighbor in $X_1,...,X_n$.
    
    Combining Lemmas~\ref{lemma:knn_distance}, \ref{lemma:sum_of_dimensions}, and \ref{lemma:pushforward_dimension} gives, for some constant $C_2 > 0$, with probability at least $1 - e^{-k/4}$,
    \begin{equation}
        \|x_0 - X_{\pi_k(x_0; \{X_1,...,X_n\})}\| \leq C_2 \left( \frac{k \sigma_X^{d_X}}{n} \right)^{\frac{1}{d_X + d_Z/s_g}}.
        \label{ineq:passive_upper_bound_noisy_bias_term}
    \end{equation}
    Now suppose that $(k/n)^{s_g/d_Z} \geq \sigma_X$. By the triangle inequality,
    \begin{equation}
        \|x_0 - X_{\pi_k(x_0; \{X_1,...,X_n\})}\|
        \leq \|x_0 - X_{\pi_k(x_0; \{g(Z_1),...,g(Z_n)\})}\| + \max_{i \in [n]} \|\epsilon_{X,i}\|.
        \label{ineq:passive_upper_low_noise_split}
    \end{equation}
    Combining Lemmas~\ref{lemma:knn_distance} and \ref{lemma:pushforward_dimension}, there exists a constant $C_4 > 0$ such that
    \begin{equation}
        \|x_0 - X_{\pi_k(x_0; \{g(Z_1),...,g(Z_n)\})}\| \leq C_4 \left( \frac{k}{n} \right)^{\frac{s_g}{d_Z}}.
        \label{ineq:passive_upper_noiseless_knn_distance}
    \end{equation}
    Meanwhile, by a standard maximal inequality for sub-Gaussian random variables (Lemma~\ref{lemma:sub_gaussian_maximal_inequality} in Appendix~\ref{app:supplementary_lemmas}), for some constant $C_3 > 0$, with probability at least $1 - \delta/2$,
    \begin{equation}
        \max_{i \in [n]} \|\epsilon_i\|
        \leq \sigma_X \sqrt{2 \log \frac{4n}{\delta}}
        \leq C_3 \left( \frac{k}{n} \right)^{s_g/d_Z} \sqrt{\log \frac{n}{\delta}}.
        \label{ineq:passive_upper_maximal_inequality}
    \end{equation}
    Plugging Inequalities~\eqref{ineq:passive_upper_noiseless_knn_distance} and \eqref{ineq:passive_upper_maximal_inequality} into Inequality~\eqref{ineq:passive_upper_low_noise_split} gives, with probability at least $1 - e^{-k/4}$,
    \begin{equation}
        \|x_0 - X_{\pi_k(x_0; \{X_1,...,X_n\})}\|
        \leq (C_3 + C_4) \left( \frac{k}{n} \right)^{\frac{s_g}{d_Z}} \sqrt{\log \frac{n}{\delta}}.
        \label{ineq:passive_upper_noiseless_bias_term}
    \end{equation}
    One can check that, whenever $(k/n)^{s_g/d_Z} < \sigma_X$, the bound~\eqref{ineq:passive_upper_bound_noisy_bias_term} dominates the bound~\eqref{ineq:passive_upper_noiseless_bias_term}. Hence, can take the maximum of these two bounds and omit the condition $(k/n)^{s_g/d_Z} \geq \sigma_X$; i.e., for some constant $C_5 > 0$, we always have
    \[\|x_0 - X_{\pi_k(x_0; \{X_1,...,X_n\})}\|
        \leq C_5 \max \left\{
            \left( \frac{k \sigma_X^{d_X}}{n} \right)^{\frac{1}{d_X + d_Z/s_g}},
            \left( \frac{k}{n} \right)^{\frac{s_g}{d_Z}} \sqrt{\log \frac{n}{\delta}}
        \right\}.\]
    
    plugging this into Inequality~\eqref{ineq:f_holder_continuity_bound} gives the desired result.
\end{proof}

\begin{remark}[Heteroscedastic Errors]
    As noted in Section~\ref{subsec:assumptions} of the main paper, we can dramatically weaken the assumptions of homoscedastic additive errors. In particular, suppose that the response noise level $\sigma_Y : \X \to [0, \infty)$ varies as a function of $x$, so that $Y = f(X) + \sigma_Y(X) \epsilon_Y$. This difference affects the above proof only in bounding Equation~\eqref{eq:passive_upper_bound_additive_noise}. Specifically, we instead have
    \begin{align*}
        \left| \hat f_{x_0,k} - \tilde f_{x_0,X_1,...,X_n,k} \right|
        & = \left| \frac{1}{k} \sum_{i = 1}^k \sigma_Y(X_{\pi_i(x_0; \{X_1,...,X_n\})}) \epsilon_{Y,\pi_i(x_0; \{X_1,...,X_n\})} \right| \\
        & \leq \left| \frac{1}{k} \sum_{i = 1}^k \sigma_Y(x_0) \epsilon_{Y,\pi_i(x_0; \{X_1,...,X_n\})} \right| \\
        & + \left| \frac{1}{k} \sum_{i = 1}^k (\sigma_Y(X_{\pi_i(x_0; \{X_1,...,X_n\})}) - \sigma_Y(x_0)) \epsilon_{Y,\pi_i(x_0; \{X_1,...,X_n\})} \right|.
    \end{align*}
    For the first term above, we get the same bound
    \[\left| \frac{1}{k} \sum_{i = 1}^k \sigma_Y(x_0) \epsilon_{Y,\pi_i(x_0; \{X_1,...,X_n\})} \right| \leq \sigma_Y(x_0) \sqrt{\frac{2}{k} \log \frac{4}{\delta}}\]
    as in the homoscedastic case. Additionally, if the function $\sigma_Y$ is smooth near $x_0$, the second term will be small. Specifically, suppose $\sigma_Y$ lies in a local H\"older ball $\C^{s_{\sigma_Y}}(\X, x_0; \Y; L_{\sigma_Y})$. Then, H\"older's inequality, a standard maximal inequality for sub-Gaussian random variables (Lemma~\ref{lemma:sub_gaussian_maximal_inequality} in Appendix~\ref{app:supplementary_lemmas}), and Eq.~\eqref{ineq:passive_upper_bound_noisy_bias_term} give, with high probability,
    \begin{align*}
        & \left| \frac{1}{k} \sum_{i = 1}^k (\sigma_Y(X_{\pi_i(x_0; \{X_1,...,X_n\})}) - \sigma_Y(x_0)) \epsilon_{Y,\pi_i(x_0; \{X_1,...,X_n\})} \right| \\
        & \leq |\sigma_Y(X_{\pi_k(x_0; \{X_1,...,X_n\})}) - \sigma_Y(x_0)| \max_{i \in [k]} \left| \epsilon_{Y,\pi_i(x_0; \{X_1,...,X_n\})} \right|
        \in O \left( L_{\sigma_Y} \left( \frac{k \sigma_X^{d_X}}{n} \right)^{\frac{s_{\sigma_Y}}{d_X + d_Z/s_g}} \sqrt{\log k} \right).
    \end{align*}
    As long as $s_{\sigma_Y} > s_f$, this additional error is asymptotically negligible compared to the smoothing bias term in Eq.~\eqref{ineq:passive_upper_bound_noisy_bias_term}.
    \label{remark:heteroscedastic_errors_app}
\end{remark}

\section{Lower Bounds for the Passive Case (Theorem~\ref{thm:passive_lower_bound})}
\label{app:passive_lower_bounds}

In this section, we prove our lower bounds on minimax risk in the passive case. Throughout this section, we assume $\Z = [0,1]^{d_Z}$ and $\X = \R^{d_X}$.

While, in the main paper, we assumed local H\"older continuity of $g$ and $f$ around specific points, we can actually prove stronger lower bounds over \emph{globally} H\"older continuous $g$ and $f$. Global H\"older continuity is defined as follows:
\begin{definition}[Global H\"older Seminorm, Space, and Ball]
    For any $s \in [0, 1]$ and metric spaces $(\Z, \rho_\Z)$ and $(\X, \rho_\X)$, the \emph{global H\"older seminorm} $\|\cdot\|_{\C^s(\Z, z_0; \X)}$ is defined for $f : \Z \to \X$ by
    \[\|f\|_{\C^s(\Z; \X)} := \sup_{z \in \Z} \|f\|_{\C^s(\Z, z; \X)},\]
    the \emph{global H\"older space} $\C^s(\Z; \X)$ is
    \[\C^s(\Z; \X) := \bigcap_{z \in \Z} \C^s(\Z, z; \X) = \{f : \Z \to \X \text{ such that } \|f\|_{\C^s(\Z; \X)} < \infty\}.\]
    For $L \geq 0$, the \emph{global H\"older ball} $\C^s(\Z; \X; L)$ is
    \[\C^s(\Z; \X; L) := \bigcap_{z \in \Z} \C^s(\Z, z_0; \X; L) = \{f : \Z \to \X  \text{ such that } \|f\|_{\C^s(\Z; \X)} \leq L\}.\]
    \label{def:global_holder}
\end{definition}
Comparing Definition~\ref{def:global_holder} with Definition~\ref{def:local_holder} (local H\"older continuity), one can easily see that global H\"older continuity implies Local H\"older continuity. Therefore, minimax lower bounds under global H\"older continuity immediately imply identical lower bounds under local H\"older continuity.

We can now state our main lower bound for the Passive case:
\begin{customthm}{\ref{thm:passive_lower_bound}}[Main Passive Lower Bound]
    In the passive case, there exist constants $c_1, c_2, n_0 > 0$, depending only on $L_g$, $s_g$, $L_f$, $s_f$, $d_X$, and $d_Z$, such that, for all $n \geq n_0$,
    \begin{align}
        \sup_{g,f,P_Z,\epsilon_X,\epsilon_Y} \Pr_D \left[ \left| \hat f_{x_0} - f(x_0) \right| \geq c_1 \Phi_P \right] \geq c_2,
        \label{ineq:passive_lower_bound_app}
    \end{align}
    where
    \begin{equation}
        \Phi_P
          = \max \left\{
            \left( \frac{1}{n} \right)^{\frac{s_fs_g}{d_Z}},
            \left( \frac{\sigma_X^{d_X}}{n} \right)^{\frac{s_f}{d_X + d_Z/s_g}},
            \left( \frac{\sigma_Y^2}{n} \right)^{\frac{s_fs_g}{2s_fs_g + d_Z}},
            \left( \frac{\sigma_X^{d_X} \sigma_Y^2}{n} \right)^{\frac{s_f}{2s_f + d_X + d_Z/s_g}}
        \right\}.
    \end{equation}
    \label{thm:passive_lower_bound_app}
\end{customthm}
The remainder of this section is devoted to proving Theorem~\ref{thm:passive_lower_bound}. Since Inequality~\eqref{ineq:passive_lower_bound_app} is a maximum of four lower bounds, we can prove each of these lower bounds separately. Proposition~\ref{prop:large_response_noise_passive_case} will provide the two terms depending on $\sigma_Y$. Proposition~\ref{prop:lower_bound_no_noise} will provide the $n^{-\frac{s_fs_g}{d_Z}}$ term. Finally, Proposition~\ref{prop:passive_lower_bound_no_response_noise} will provide the term depending on $\sigma_X$ but not on $\sigma_Y$.

We begin by construct a family $\Theta_{g,M}$ of functions $g : \Z \to \X$ that will be used throughout the proofs:

\begin{lemma}
    For any $M \in \mathbb{N}$, $s_g, L_g > 0$, there exists a family $\Theta_{g,M} \subseteq \mathcal{C}^{s_g}(\Z; \X; L_g)$ of globally $s_g$-H\"older continuous functions from $\Z$ to $\X$ such that:
    \begin{enumerate}
        \item $|\Theta_{g,M}| = M^{d_Z}$.
        \item for each $g \in \Theta_{g,M}$, there exists $z_g \in \Z$ with $g(z_g) = x_0$.
        \item for some $g_0 \in \X$ and some constant $C_{d_Z,s_g} > 0$ depending only on $d_Z$ and $s_g$,
        \begin{enumerate}
            \item $\|g_0 - x_0\|_\infty = C_{d_Z,s_g} L_g M^{-s_g}$.
            \item for $g \neq g' \in \Theta_{g,M}$, the functions $x \mapsto 1\{g(x) \neq g_0\}$, $x \mapsto 1\{g'(x) \neq g_0\}$ have disjoint supports in $\Z$.
        \end{enumerate}
    \end{enumerate}
    \label{lemma:Theta_g}
\end{lemma}
Intuitively, the family $\Theta_{g,M}$ is constructed such that, for any $z \in \Z$, there is no more than one $g \in \Theta_{g,M}$ such that $\|g(z) - x_0\|_\infty \lesssim M^{-s_g}$; hence, for the typical $g \in \Theta_{g,M}$, $\lesssim 1/|\Theta_{g,M}|$ of the samples $X_1,...,X_n$ will lie within $\asymp M^{-s_g}$ of $x_0$. The formal construction of $\Theta_{g,M}$ is as follows:
\begin{proof}
    For any positive integer $d$, let $K_d : \R^d \to \R$ denote the standard $d$-dimensional ``bump function'' defined by
    \begin{equation}
        K_d(x) = \exp \left( 1 - \frac{1}{1 - \|x\|_2^2} \right) 1\{\|x\|_2 \leq 1\}
        \quad \text{ for all } x \in \R^d.
        \label{eq:bump_function}
    \end{equation}
    Note in particular that $K_d$ satisfies the following:
    \begin{enumerate}
        \item $K_d \in \mathcal{C}^\infty(\R^d)$; i.e., $K_d$ is infinitely differentiable.
        \item For all $x \notin (-1, 1)^d$, $K_d(x) = 0$.
        \item $\|K_d\|_\infty = K(0) = 1$.
    \end{enumerate}
    For any $m \in [M]^{d_Z}$, define $g_m : \Z \to \X$ by
    \begin{equation}
        g_m(z) = x_0 + \frac{L_g}{M^{s_g} \|K_{d_Z}\|_{\C^{s_g}}} \left( K_{d_Z}(0) - K_{d_Z} \left( (M + 1) z - m \right) \right) v,
        \quad \text{ for all } z \in \Z,
        \label{eq:g_m}
    \end{equation}
    where $v$ is an arbitrary unit vector in $\mathbb{R}^{d_X}$. Then, define $\Theta_{g,M} := \{g_m : m \in [M]^{d_Z}\}$. Clearly, $|\Theta_{g,M}| = M^{d_Z}$. By Eq.~\eqref{eq:g_m}, $\Theta_{g,M} \subseteq \C^{s_g}(\Z; \X; L_g)$. Moreover, for each $m \in [M]^d$, $g_m(m/(M+1)) = x_0$, while, whenever $\|(M + 1)z - m\|_\infty \geq 1$, $g_m(z) = g_0 := x_0 + \frac{L_g}{\|K_{d_Z}\|_{\C^{s_g}}} M^{-s_g} v$.
\end{proof}

Next, we present a general information-theoretic lower bound in the ``stochastic case'', i.e., in the presence of response noise ($\sigma_Y > 0$). Together with the construction of the family $\Theta_{g,M}$ in Lemma~\ref{lemma:Theta_g}, this lemma will be the basis for our lower bounds depending on $\sigma_Y$, in both the active and passive cases:

\begin{lemma}[General Lower Bound with Response Noise]
    Suppose $Y_i|X_i \sim \mathcal{N}(f(X_i), \sigma_Y^2)$, and suppose that the distribution of $(Z_i,X_i)$ depends only on the preceding observations $\{(Z_j,X_j,Y_j)\}_{j = 1}^{i -1}$. Let $h > 0$, and let $P_g$ be any probability distribution over $\C^{s_g}(\Z;\X;L_g)$. Then, there exist $f_0, f_1 \in \C^{s_f}(\X; \Y; L_f)$ such that, if $P_0 = P_g \times \delta_{f_0}$ is the product distribution of $P_g$ and a unit mass on $f_0$, and $P_1$ is the product distribution of $P_g$ and a unit mass on $f_1$, for some constant $C_{d_X,s_f} > 0$ depending only on $d_X$ and $s_f$,
    \begin{align}
        \notag
        \inf_{\hat f_{x_0}} \sup_{g,f,P_Z,\epsilon_X,\epsilon_Y} & \Pr_D \left[ |\hat f_{x_0} - f(x_0)| \geq C_{d_X,s_f} L_f h^{s_f} \right] \\
        \label{ineq:general_stochastic_lower_bound}
        & \geq \frac{1}{2} \left( 1 - C_{d_X,s_f}\frac{L_f h^{s_f}}{\sigma_Y} \sqrt{\sum_{i = 1}^n \E_{D, g \sim P_0} \left[ 1\{\|X_i - x_0\|_\infty < h\} \right]} \right).
    \end{align}
    Here, the infimum is taken over all estimators $\hat f_{x_0}$ (i.e., all functions $\hat f_{x_0} : \left( \Z \times \X \times \Y \right)^n \to \Y$, the supremum is taken over all $g \in \C^{s_g}(\Z;\X;L_g)$, $f \in \C^{s_f}(\X; \Y; L_f)$, sub-Gaussian covariate noise random variables $\epsilon_X$ with dimension $d_X$ around $0$, and sub-Gaussian response noise random variables $\epsilon_Y$, and we used the shorthand $D := \{(Z_i,X_i,Y_i)\}_{i = 1}^n$ to denote the entire dataset.
    \label{lemma:TV_upper_bound}
\end{lemma}
Intuitively, Lemma~\ref{lemma:TV_upper_bound} provides a lower bound in terms of the expected number of covariate observations $X_i$ that are ``informative'' about $f(x_0)$, i.e., those that lie sufficiently near $x_0$. The family $\Theta_{g,M}$ was constructed in Lemma~\ref{lemma:Theta_g} such that 
Hence, it remains only to bound the expected number of such $X_i$, which will be small for $g$ lying in the family $\Theta_{g,M}$.
Note that inequality~\eqref{ineq:general_stochastic_lower_bound} applies in both active and passive settings. However, in the passive setting, the distribution of each $Z_i$ will be independent of $g$, whereas, in the active setting, $Z_i$ may depend on $g$ through the preceding observations (specifically, $(Z_1,X_1),...,(Z_{i-1},X_{i-1})$).

The proof of Lemma~\ref{lemma:TV_upper_bound} is based on the method of two fuzzy hypotheses, specifically the following lemma:
\begin{lemma}[Theorem 2.15 of \citet{tsybakov2008introduction}, Total Variation Case]
    Let $P_0$ and $P_1$ be probability distributions over $\Theta$, and let $F : \Theta \to \R$. Suppose
    \begin{equation}
        \sup_{\theta \in \supp(P_0)} F(\theta) \leq 0, \quad \text{ and } \quad s := \inf_{\theta \in \supp(P_1)} F(\theta) > 0.
        \label{condition:min_distance}
    \end{equation}
    Then,
    \[\inf_{\hat F} \sup_{\theta \in \Theta} \Pr_{\theta} \left[ |\hat F - F(\theta)| \geq s/2 \right]
      \geq \frac{1 - D_{\text{TV}}(P_0, P_1)}{2}.\]
    \label{lemma:tsybakov_two_fuzzy_hypotheses}
\end{lemma}
The method of two fuzzy hypotheses is one of several standard methods for deriving minimax lower bounds on statistical estimation error (see Chapter 2 of \citet{tsybakov2008introduction} for further discussion of this and other methods). The method of two fuzzy hypotheses, in particular, is typically used in semiparametric problems where one wishes to estimate a parametric quantity (here, $\hat f(x_0)$) that is subject to variation in a nonparametric ``nuisance'' quantity (here, $g$). The key idea of the method is to lower bound the supremum over the nuisance quantity (here, $\sup_g$) by in terms of an appropriate information-theoretic measure (here $D_{\text{TV}}$) computed over an appropriately chosen distribution of the nuisance variable (here, $P_g$). By choosing this distribution carefully (as we will do later in this section, using the family $\Theta_{g,M}$ constructed in Lemma~\ref{lemma:Theta_g}), the information-theoretic measure will give a tight lower bound on the probability of error.

We now present the proof of Lemma~\ref{lemma:TV_upper_bound}.
\begin{proof}
    Let $h > 0$. Define $f_0, f_1 : \X \to \Y$ by
    \begin{equation}
        f_0(x) = 0 \quad \text{ and } \quad f_1(x) = \frac{L_f h^{s_f}}{\|K_{d_X}\|_{\C^{s_f}}} K_{d_X} \left( \frac{x - x_0}{h} \right)
        \quad \text{ for all } x \in \X,
        \label{eq:f0_and_f1}
    \end{equation}
    where $K_{d_X}$ is as defined in Eq.~\eqref{eq:bump_function}.
    Clearly,
    \[\sup_{(g, f) \in \supp(P_0)} f(x_0) = 0,\]
    while
    \begin{equation}
        s := \inf_{(g, f) \in \supp(P_1)} f(x_0) = \frac{L_f h^{s_f}}{\|K_{d_X}\|_{\C^{s_f}}} K_{d_X}(0)
        = \frac{L_f h^{s_f}}{\|K_{d_X}\|_{\C^{s_f}}},
        \label{eq:s}
    \end{equation}
    satisfying the condition~\eqref{condition:min_distance}.
    It remains only to bound $D_{\text{TV}}(P_0, P_1)$.
    For each $i \in [n]$, let $P_{0,i}$ and $P_{1,i}$ denote the conditional distributions of
    \[(Z_i,X_i,Y_i)
        \quad \text{ given } \quad
        (Z_1, X_1, Y_1),...,(Z_{i-1},X_{i-1},Y_{i-1})\]
    under $P_0$ and $P_1$, respectively, and let $P_{Y,0,i}$ and $P_{Y,1,i}$ denote the conditional distributions of
    \[Y_i
        \quad \text{ given } \quad
        (Z_1, X_1, Y_1),...,(Z_{i-1},X_{i-1},Y_{i-1}),(Z_i,X_i)\]
    under $P_0$ and $P_1$, respectively. By a standard formula for the KL divergence between two Gaussian random variables,
    \[D_{\text{KL}} \left( P_{Y,0,i} \middle\| P_{Y,1,i} \right)
      = \frac{f_1^2(X_i)}{2\sigma_Y^2}.\]
    Moreover, by construction of $P_0$ and $P_1$, the conditional distributions of
    \[(Z_i,X_i)
        \quad \text{ given } \quad
        (Z_1,X_1,Y_1),...,(Z_{i-1},X_{i-1},Y_{i-1})\]
    are identical under $P_0$ and $P_1$.
    Therefore, by the chain rule for KL divergence,
    \begin{align*}
        D_{\text{KL}} \left( P_0, P_1 \right)
        & = \sum_{i = 1}^n \E_{D, g \sim P_0} \left[ D_{\text{KL}} \left( P_{0,i} \middle\| P_{1,i} \right) \right] \\
        & = \sum_{i = 1}^n \E_{D, g \sim P_0} \left[ D_{\text{KL}} \left( P_{Y,0,i} \middle\| P_{Y,1,i} \right) \right] \\
        & = \frac{1}{2\sigma_Y^2} \sum_{i = 1}^n \E_{D, g \sim P_0} \left[ f_1^2(X_i) \right] \\
        & \leq \frac{\|f_1\|_\infty^2}{2\sigma_Y^2} \sum_{i = 1}^n \E_{D, g \sim P_0} \left[ 1\{f_1(X_i) \neq 0\} \right] \\
        & = \frac{L_f^2 h^{2s_f}}{2 \sigma_Y^2 \|K_{d_X}\|_{\C^{s_f}}^2} \sum_{i = 1}^n \E_{D, g \sim P_0} \left[ 1\{f_1(X_i) \neq 0\} \right] \\
        & \leq \frac{L_f^2 h^{2s_f}}{2 \sigma_Y^2 \|K_{d_X}\|_{\C^{s_f}}^2} \sum_{i = 1}^n \E_{D, g \sim P_0} \left[ 1\{\|X_i - x_0\|_\infty < h\} \right],
    \end{align*}
    where the last inequality follows from the fact that $f_1$ is non-zero only on the open $\ell_\infty$ ball of radius $h$ centered at $x_0$ (see Eq.~\eqref{eq:f0_and_f1}). Thus, by Pinsker's Inequality,
    \begin{align*}
        D_{\text{TV}} \left( P_0, P_1 \right)
        \leq \sqrt{2 D_{\text{KL}} \left( P_0, P_1 \right)}
        & \leq \frac{L_f h^{s_f}}{\sigma_Y \|K_{d_X}\|_{\C^{s_f}}} \sqrt{\sum_{i = 1}^n \E_{D, g \sim P_0} \left[ 1\{\|X_i - x_0\|_\infty < h\} \right]}
    \end{align*}
    The result now follows by plugging this bound on $D_{\text{TV}} \left( P_0, P_1 \right)$ into Lemma~\ref{lemma:tsybakov_two_fuzzy_hypotheses}.
\end{proof}

We now utilize Lemma~\ref{lemma:TV_upper_bound} to prove the following minimax lower bound for the passive case under response noise:

\begin{proposition}[Minimax Lower Bound, Passive Case, $\sigma_Y$ large]
    In the passive case, there exists a constant $C > 0$ depending only on $d_Z$, $d_X$, $s_g$, $L_g$, $s_f$, $L_f$ such that
    \[\sup_{g,f,P_Z,\epsilon_X,\epsilon_Y} \Pr_D \left[ \left| \hat f_{x_0} - f(x_0) \right| \geq C \max \left\{ \left( \frac{\sigma_Y^2}{n} \right)^{\frac{s_fs_g}{2s_fs_g + d_Z}},
    \left(
            \frac{\sigma_X^{d_X} \sigma_Y^2}{n} \right)^{\frac{s_f}{2s_f + d_X + d_Z/s_g}}
            \right\} \right] \geq \frac{1}{4}.\]
    \label{prop:large_response_noise_passive_case}
\end{proposition}

\begin{proof}
    Consider the case in which $\sigma_X \epsilon_{X,i}$ is uniformly distributed over the cube $[-\frac{\sigma_X}{\sqrt{d_X}},\frac{\sigma_X}{\sqrt{d_X}}]^{d_X}$. Let
    \begin{equation}
        h := \left\{
        \begin{array}{cc}
            \left( \frac{\|K_{d_X}\|_{\C^{s_f}}^2}{2^{2 + d_Z/s_g}} \frac{L_g^{d_Z/s_g}}{L_f^2} \frac{\sigma_Y^2}{n} \right)^{\frac{1}{2s_f + d_Z/s_g}} : & \text{ if } \sigma_X \leq d_X^{-1/2} \left( \frac{\|K_{d_X}\|_{\C^{s_f}}^2}{2^{2 + d_Z/s_g}} \frac{L_g^{d_Z/s_g}}{L_f^2} \frac{\sigma_Y^2}{n} \right)^{\frac{1}{2s_f + d_Z/s_g}} \\
            \left( \frac{\|K_{d_X}\|_{\C^{s_f}}^2}{2^{2 + d_Z/s_g} d_X^{d_X/2}} \frac{L_g^{d_Z/s_g}}{L_f^2} \frac{\sigma_X^{d_X} \sigma_Y^2}{n} \right)^{\frac{1}{2s_f + d_X + d_Z/s_g}} : & \text{ else }
        \end{array}
        \right.
        \label{eq:low_noise_h}
    \end{equation}
    and
    \begin{equation}
        M := \left( \frac{L_g}{\|K_{d_Z}\|_{\C^{s_g}}\left( h + \frac{\sigma_X}{\sqrt{d_X}} \right)} \right)^{1/s_g}.
        \label{eq:low_noise_M}
    \end{equation}
    Suppose $P_g$ is the uniform distribution over $\Theta_{g,M}$, given in Lemma~\ref{lemma:Theta_g}.
    
    By the Law of Total Probability,
    \begin{align}
        \E_{D, g \sim P_0} \left[ 1\{\|X_i - x_0\|_\infty < h\} \right]
          \notag
          & = \Pr_{D, g \sim P_0} \left[ \|g(Z_i) + \sigma_X \epsilon_{X,i} - x_0\|_\infty < h \right] \\
          \notag
          & \leq \Pr_{D, g \sim P_0} \left[ \|g(Z_i) + \sigma_X \epsilon_{X,i} - x_0\|_\infty < h \middle| g(Z_i) \neq g_0 \right] \Pr_{D, g \sim P_0}[g(Z_i) \neq g_0] \\
          \label{ineq:low_noise_split_probability}
          & + \Pr_{D, g \sim P_0} \left[ \|g(Z_i) + \sigma_X \epsilon_{X,i} - x_0\|_\infty < h \middle| g(Z_i) = g_0 \right]
    \end{align}
    Since in the passive case, $g$ is uniformly distributed over $\Theta_{g,M}$ independently of $Z_i$ and the elements of $\Theta_{g,M}$ have disjoint supports,
    \[\Pr_{D, g \sim P_0}[g(Z_i) \neq g_0]
      \leq \frac{1}{|\Theta_{g,M}|}
      = M^{-d_Z}
      = \left( \frac{\|K_{d_Z}\|_{\C^{s_g}}}{L_g} \left( h + \frac{\sigma_X}{\sqrt{d_X}} \right) \right)^{d_Z/s_g}.\]
    Since $\epsilon_{X,i}$ has a uniform distribution with mean $0$, the function $x \mapsto \Pr[\|\sigma_X \epsilon_{X,i} + x\| \leq h]$ is maximized at $x = 0$, and so, similar to Inequality~\eqref{ineq:uniform_density},
    \begin{align*}
        \Pr_{D, g \sim P_0} \left[ \|g(Z_i) + \sigma_X \epsilon_{X,i} - x_0\|_\infty < h \middle| g(Z_i) \neq g_0 \right]
        & \leq \Pr_{D, g \sim P_0} \left[ \|\sigma_X \epsilon_{X,i}\|_\infty < h \middle| g(Z_i) \neq g_0 \right] \\
        & = \min \left\{ 1, \left( \frac{h\sqrt{d_X}}{\sigma_X} \right)^{d_X} \right\}.
    \end{align*}
    Finally, since $\sigma_X \epsilon_{X,i}$ is uniformly distributed over the cube $[-\frac{\sigma_X}{\sqrt{d_X}},\frac{\sigma_X}{\sqrt{d_X}}]^{d_X}$, by our choice (Eq.~\eqref{eq:low_noise_M}) of $M$, if $g(Z_i) = g_0$, then there is no overlap between the supports of $X_i$ and $f_1$; formally,
    \begin{align*}
        \Pr_{D, g \sim P_0} \left[ \|g(Z_i) + \sigma_X \epsilon_{X,i} - x_0\|_\infty < h \middle| g(Z_i) = g_0 \right]
        & = \left( \max \left\{0, x_0 + h - g_0 + \frac{\sigma_X}{\sqrt{d_X}}\right\} \right)^{d_X} \\
        & = \left( \max \left\{0, h + \frac{\sigma_X}{\sqrt{d_X}} - \frac{L_g}{M^{s_g} \|K_{d_Z}\|_{\C^{s_g}}} \right\} \right)^{d_X}
          = 0.
    \end{align*}
    Plugging these into Inequality~\eqref{ineq:low_noise_split_probability} gives
    \begin{align}
        \E_{D, g \sim P_0} \left[ 1\{\|X_i - x_0\|_\infty < h\} \right]
          & \leq \min \left\{ 1, \left( \frac{h\sqrt{d_X}}{\sigma_X} \right)^{d_X} \right\} \left( \frac{\|K_{d_Z}\|_{\C^{s_g}}}{L_g} \left( h + \frac{\sigma_X}{\sqrt{d_X}} \right) \right)^{d_Z/s_g},
          \label{ineq:min_split}
    \end{align}
    and plugging this into Lemma~\eqref{lemma:TV_upper_bound} gives
    \begin{align*}
        \inf_{\hat f_{x_0}} \sup_{g,f,P_Z,\epsilon_X,\epsilon_Y} & \Pr_D \left[ |\hat f_{x_0} - f(x_0)| \geq C h^{s_f} \right] \\
        & \geq \frac{1}{2} \left( 1 - C \frac{h^{s_f}}{\sigma_Y} \sqrt{n \min \left\{ 1, \left( \frac{h\sqrt{d_X}}{\sigma_X} \right)^{d_X} \right\} \left( h + \frac{\sigma_X}{\sqrt{d_X}} \right)^{d_Z/s_g}} \right),
    \end{align*}
    where $C > 0$ is a constant depending only on $d_X$, $L_f$, $s_f$, $d_Z$, $L_g$, and $s_g$.
    Plugging in the appropriate value of $h$, according to Eq.~\eqref{eq:low_noise_h}, depending on whether $\sigma_X$ is greater or less than $h/\sqrt{d_X}$, gives
    \[\inf_{\hat F} \sup_{\theta \in \Theta} \Pr_{\theta} \left[ |\hat F - F(\theta)| \geq s/2 \right]
      \geq \frac{1}{4},\]
    where
    \[s \asymp \left\{
        \begin{array}{cc}
            \left( \frac{\sigma_Y^2}{n} \right)^{\frac{s_fs_g}{2s_fs_g + d_Z}} :
                & \text{ if } \sigma_X \sqrt{d_X} \leq \left( \frac{1}{2^{2 + d_Z/s_g}} \frac{L_g^{d_Z/s_g}}{L_f^2} \frac{\|K_{d_X}\|_{\C^{s_f}}^2}{\|K_{d_Z}\|_{\C^{s_g}}^{s_f/s_g}} \frac{\sigma_Y^2}{n} \right)^{\frac{1}{2s_f + d_Z/s_g}} \\
            \left(
            \frac{\sigma_X^{d_X} \sigma_Y^2}{n} \right)^{\frac{s_f}{2s_f + d_X + d_Z/s_g}} : & \text{ else }
        \end{array}
    \right..\]
\end{proof}

Proposition~\ref{prop:large_response_noise_passive_case} provided minimax lower bounds for the passive case under response noise of scale $\sigma_Y > 0$, based on information-theoretic bounds showing that the information obtainable from a single pair $(X_i,Y_i)$ decays as $\sigma_Y^{-2}$; hence, a different approach is needed to obtain tight bounds when $\sigma_Y$ is very small or $0$. Our next results provide tight lower bounds for the case when $\sigma_Y$ is small or $0$ using a combinatorial, rather than information-theoretic, approach. In particular, we show that, given a sufficiently large family $\Theta_{g,M}$ of functions $g$ with $x \mapsto 1\{g(x) \neq g_0\}$ having disjoint supports, the probability of observing any $X_i$ with $g(X_i) \neq g_0$ is small.

\begin{proposition}[Minimax Lower Bound, Active and Passive Case, $\sigma_X$ small, $\sigma_Y$ small]
    In both active and passive cases,
    \begin{equation}
        \inf_{\hat f_{x_0}} \sup_{g,f,P_Z,\epsilon_X,\epsilon_Y}
            \Pr_D \left[ \left| \hat f_{x_0} - f(x_0) \right|
                         \geq C
                              \left( \frac{1}{n} \right)^{\frac{s_fs_g}{d_Z}}
            \right]
        \geq \frac{1}{4},
    \end{equation}
    where $C$ is a constant (given in Eq.~\eqref{ineq:lower_bound_no_noise}) depending only on $L_g$, $s_g$, $L_f$, $s_f$, $d_X$, and $d_Z$.
    \label{prop:lower_bound_no_noise}
\end{proposition}
We note that, in contrast to other the other lower bounds in this section, Proposition~\ref{prop:lower_bound_no_noise} applies to to both the active and passive settings, as its proof makes no assumptions on the distribution of the control variables $Z_1,...,Z_n$.

\begin{proof}
    Let $\Theta_{g,M}$ be as given in Lemma~\ref{lemma:Theta_g}, with $M := \lceil (2n)^{1/d_Z} \rceil$. Additionally, let $f_0$ and $f_1$ be as in Eq.~\eqref{eq:f0_and_f1}.
    For each $g \in \Theta_{g,M}$, let $E_g$ denote the event $g(Z_1) = g(Z_2) = \cdots = g(Z_n) = g_0$. Since, for any $g_1 \neq g_2 \in \Theta_{g,M}$, the functions $z \mapsto 1\{g_1(z) \neq g_0\}$ and $z \mapsto 1\{g_2(z) \neq g_0\}$ have disjoint support, for any $Z_1,...,Z_n$,
    \[\sum_{g \in \Theta_{g,M}} 1 - 1_{E_g} \leq n.\]
    Since, by our choice of $M$, $|\Theta_{g,M}| = M^{d_Z} \geq 2n$, for any (potentially randomized) estimator $\hat f_{x_0}$,
    \begin{align*}
        |\Theta_{g,M}| \max_{g \in \Theta_{g,M}} \Pr[E_g]
        & \geq \sum_{g \in \Theta_{g,M}} \Pr[E_g] \\
        & = \E \left[ \sum_{g \in \Theta_{g,M}} 1_{E_g} \right] \\
        & = \E \left[ \sum_{g \in \Theta_{g,M}} 1 - (1 - 1_{E_g}) \right]
      \geq |\Theta_{g,M}| - n \geq n,
    \end{align*}
    and so there exists $g_{\hat f_{x_0}} \in \Theta_{g,M}$ such that
    \[\Pr \left[ E_{g_{\hat f_{x_0}}} \middle| g = g_{\hat f_{x_0}} \right]
      \geq \frac{n}{|\Theta_{g,M}|}
      \geq \frac{n}{((2n)^{1/d_Z} + 1)^{d_Z}}
      \geq 2^{-d_Z - 1}.\]
    For $h := \|x_0 - g_0\|_\infty = \frac{L_g}{M^{s_g} \|K_{d_Z}\|_{\C^{s_g}}}$, under the event $E_g$, the distribution of the data $\{(X_i, Y_i, Z_i)\}_{i = 1}^n$ is independent of whether $f = f_0$ or $f = f_1$. In particular, either
    \[\Pr \left[ \hat f_{x_0} \geq \frac{f_0(x_0) + f_1(x_0)}{2} \middle| f = f_0 \right] \geq \frac{1}{2}
        \quad \text{ or } \quad
        \Pr \left[\hat f_{x_0} \leq \frac{f_0(x_0) + f_1(x_0)}{2} \middle| f = f_1\right]
        \geq \frac{1}{2},\]
    and so
    \[\max_{g \in \Theta_{g,M}, f \in \{f_0,f_1\}} \Pr \left[ \left| \hat f_{x_0} - f(x_0) \right|       \geq \frac{|f_1(x_0) - f_0(x)|}{2} \right]
        \geq \frac{1}{2} \Pr[E_{g_{\hat f_{x_0}}}]
        \geq 2^{-d_Z - 2}.\]
    By construction of $f_0$ and $f_1$,
    \[|f_1(x_0) - f_0(x_0)| = \frac{L_f h^{s_f}}{\|K_{d_X}\|_{\C^{s_f}}},\]
    and so plugging in our choices of $M$ and $h$ gives
    \begin{equation}
        \max_{g \in \Theta_{g,M}, f \in \{f_0,f_1\}} \Pr \left[ \left| \hat f_{x_0} - f(x_0) \right|       \geq \frac{L_f L_g^{s_f}}{2\|K_{d_X}\|_{\C^{s_f}} \|K_{d_Z}\|_{\C^{s_g}}^{s_f} \left( (2n)^{s_fs_g/d_Z} + 1 \right)} \right]
        \geq 2^{-d_Z - 2}.
        \label{ineq:lower_bound_no_noise}
    \end{equation}
\end{proof}

\begin{proposition}[Minimax Lower Bound, Passive Case, $\sigma_X$ large, $\sigma_Y$ small]
    In the passive case, for some constant $C > 0$, whenever $n \geq C \sigma_X^{-d_Z/s_g}$,
    \begin{equation}
        \inf_{\hat f_{x_0}} \sup_{g,f,P_Z,\epsilon_X,\epsilon_Y}
            \Pr_D \left[ \left| \hat f_{x_0} - f(x_0) \right|
                         \geq C
                              \left( \frac{\sigma_X^{d_X}}{n} \right)^{\frac{s_f}{d_X + d_Z/s_g}}
            \right]
        \geq \frac{1}{2e} > 0,
    \end{equation}
    where $C > 0$ (given below in Eq.~\eqref{eq:passive_lower_bound_no_response_noise_constant}) depends only on $L_g$, $s_g$, $L_f$, $s_f$, $d_X$, and $d_Z$.
    \label{prop:passive_lower_bound_no_response_noise}
\end{proposition}

We note that the condition $n \gtrsim \sigma_X^{-d_Z/s_g}$ is precisely when the rate $\left( \sigma_X^{d_X}/n \right)^{\frac{s_f}{d_X + d_Z/s_g}}$ in Proposition~\ref{prop:passive_lower_bound_no_response_noise} dominates the rate $n^{\frac{-s_fs_g}{d_Z}}$ obtained earlier in Proposition~\ref{prop:lower_bound_no_noise}; hence, this condition can be omitted when we take the maximum of these lower bounds in Theorem~\ref{thm:passive_lower_bound}.

\begin{proof}
    Suppose $Z$ is uniformly distributed over $\Z$ and $\sigma_X \epsilon_{X,i}$ is uniformly distributed on the set $\{0\} \times [-\frac{\sigma_X}{\sqrt{d_X}},\frac{\sigma_X}{\sqrt{d_X}}]^{d_X} \subseteq \R^{d_X + 1}$ (i.e., the first coordinate of $\epsilon_{X_i}$ is always $0$ and the remaining $d_X$ coordinates are uniformly random). Define $g : \Z \to \X$ by $g(z) = L_g \|z\|_\infty^{s_g} e_1$, where $e_1 \in \R^{d_X+1}$ denotes the first canonical basis vector. It is straightforward to verify that (a) $Z$ has dimension $d_Z$ around $0$, (b) $\epsilon_{X,i}$ is sub-Gaussian, (c) $\epsilon_{X,i}$ has dimension $d_X$ around $0$, and (d) $g \in \C^{s_g}(\Z; \X)$. Let $f_0$ and $f_1$ be as in Eq.~\eqref{eq:f0_and_f1} with $x_0 = 0$.
    
    Let $r^* := \min \{L_g, \frac{\sigma_X}{\sqrt{d_X}}\}$. Then, for each $i \in [n]$ and any $r \in [0, r^*]$, since $Z_i$ and $\epsilon_{X,i}$ are independent,
    \begin{align*}
        \Pr \left[ \|X_i\|_\infty \leq r \right]
        & = \Pr \left[ \max\{\|g(Z_i)\|, \|\epsilon_i\|_\infty\} \leq r \right] \\
        & = \Pr \left[ \|Z_i\|_\infty \leq (r/L_g)^{1/s_g} \right] \Pr \left[ \|\epsilon_i\|_\infty \leq r \right] \\
        & = \left( \frac{r}{L_g} \right)^{d_Z/s_g} \left( \frac{r \sqrt{d_X}}{\sigma_X} \right)^{d_X}.
    \end{align*}
    Since, in the passive case, $X_1,...,X_n$ are IID,
    \begin{align*}
        \Pr \left[ \min_{i \in [n]} \|X_i\|_\infty \geq r \right]
        & = \left( \Pr \left[ \|X_i\|_\infty \geq r \right] \right)^n \\
        & = \left( 1 - \Pr \left[ \|X_i\|_\infty \leq r \right] \right)^n \\
        & = \left( 1 - \left( \frac{r}{L_g} \right)^{d_Z/s_g} \left( \frac{r \sqrt{d_X}}{\sigma_X} \right)^{d_X} \right)^n.
    \end{align*}
    For $n$ sufficiently large that $r \leq r^*$ (i.e., $n \geq C \sigma_X^{-d_Z/s_g}$, as assumed), plugging in
    \begin{equation}
        r = \left( \left( \frac{\sigma_X}{\sqrt{d_X}} \right)^{d_X} \frac{L_g^{d_Z/s_g}}{n} \right)^{\frac{1}{d_X + d_Z/s_g}}
        \label{eq:min_nearest_neighbor_radius_passive_case}
    \end{equation}
    gives
    \[\Pr \left[ \min_{i \in [n]} \|X_i\|_\infty \geq \left( \left( \frac{\sigma_X}{\sqrt{d_X}} \right)^{d_X} \frac{L_g^{d_Z/s_g}}{n} \right)^{\frac{1}{d_X + d_Z/s_g}} \right]
        = \left( 1 - \frac{1}{n} \right)^n
        \geq 1/e > 0,\]
    since $\left( 1 - \frac{1}{n} \right)^n$ approaches $1/e$ from above as $n \to \infty$.
    
    Now let $f_0$ and $f_1$ be as in Eq.~\eqref{eq:f0_and_f1}, with $h = r$ above. Given the event $\min_{i \in [n]} \|X_i\|_\infty \geq h$, the distribution of the data $\{(X_i,Y_i,Z_i)\}_{i = 1}^n$ is independent of whether $f = f_0$ or $f = f_1$. Hence, as argued in the proof of the previous Proposition~\eqref{prop:lower_bound_no_noise},
    \[\max_{f \in \{f_0,f_1\}} \Pr \left[ \left| \hat f_{x_0} - f(x_0) \right| \geq \frac{|f_1(x_0) - f_0(x)|}{2} \right]
        \geq \frac{1}{2} \Pr \left[ \min_{i \in [n]} \|X_i\|_\infty \leq h \right]
        \geq \frac{1}{2e}.\]
    Since, by construction of $f_0$ and $f_1$,
    \[|f_1(x_0) - f_0(x_0)| = \frac{L_f h^{s_f}}{\|K_{d_X}\|_{\C^{s_f}}},\]
    it follows that
    \[\max_{f \in \{f_0,f_1\}} \Pr \left[ \left| \hat f_{x_0} - f(x_0) \right| \geq \frac{L_f h^{s_f}}{2\|K_{d_X}\|_{\C^{s_f}}} \right]
        \geq \frac{1}{2e}.\]
    Plugging in our choice of $h = r$ from Eq~\eqref{eq:min_nearest_neighbor_radius_passive_case} above gives
    \[\max_{f \in \{f_0,f_1\}} \Pr \left[ \left| \hat f_{x_0} - f(x_0) \right| \geq C \left( \frac{\sigma_X}{n} \right)^{\frac{s_f}{d_X + d_Z/s_g}} \right]
        \geq \frac{1}{2e},\]
    where
    \begin{equation}
        C = \frac{L_f \left( d_X^{-d_X/2} L_g^{d_Z/s_g} \right)^{\frac{s_f}{d_X + d_Z/s_g}}}{2\|K_{d_X}\|_{\C^{s_f}}}.
        \label{eq:passive_lower_bound_no_response_noise_constant}
    \end{equation}
\end{proof}

\section{Upper Bound for the Active Case (Theorem~\ref{thm:active_upper_bound})}
\label{app:active_upper_bound}

In this section, we state and prove our main upper bound for the two-stage active learner presented in Algorithm~\ref{alg:active_learning_algorithm}. We first state the full result:
\begin{customthm}{\ref{thm:active_upper_bound}}[Main Active Upper Bound]
    Let $\delta \in (0, 1)$ Let $\hat f_{x_0}$ denote the estimator proposed in Algorithm~\ref{alg:active_learning_algorithm}, with $k, \ell \geq 4 \log \frac{6}{\delta}$. Then, for some constant $C > 0$ depending only on $L_g$, $s_g$, $L_f$, $s_f$, $d_X$, and $d_Z$, with probability at least $1 - \delta$,
    \[\left| \hat f_{x_0,k} - f(x_0) \right|
        \leq C_4 \left( \sigma_Y \sqrt{\frac{1}{k} \log \frac{1}{\delta}}
        + \sigma_X^{s_f} \left( \frac{k}{n} \right)^{\frac{s_f}{d_X}}
        + \left( \frac{\ell}{n} \right)^{\frac{s_fs_g}{d_Z}} + \left( \sigma_X \sqrt{\frac{1}{\ell} \log \frac{n}{\delta}} \right)^{s_f} \right),\]
    In particular, setting
    \begin{equation}
        k \asymp \max \left\{
            \log \frac{1}{\delta},
            \min \left\{
                n,
                \left( \frac{\sigma_Y}{\sigma_X^{s_f}} \right)^{\frac{2d_X}{2s_f + d_X}} n^{\frac{2s_f}{2s_f + d_X}}
            \right\}
        \right\}
        \label{eq:active_k}
    \end{equation}
    and
    \begin{equation}
        \ell \asymp \max \left\{
            \log \frac{1}{\delta},
            \min \left\{
                n,
                \sigma_X^{\frac{2d_Z}{2s_g + d_Z}} n^{\frac{2s_g}{2s_g + d_Z}} (\log n)^{\frac{d_Z}{2s_g + d_Z}}
            \right\}
        \right\},
        \label{eq:active_ell}
    \end{equation}
    we have, with probability at least $1 - \delta$, $\left| \hat f_{x_0} - f(x_0) \right| \leq C \Phi_A^* \log \frac{1}{\delta}$, where
    \[\Phi_A^* := \max \left\{ \left( \frac{1}{n} \right)^{\frac{s_fs_g}{d_Z}},
        \left( \frac{\sigma_X^{d_X}}{n} \right)^{\frac{s_f}{d_X}},
        \frac{\sigma_Y}{\sqrt{n}},
        \left( \frac{\sigma_Y^2 \sigma_X^{d_X}}{n} \right)^{\frac{s_f}{2s_f + d_X}},
        \left( \frac{\sigma_X^2 \log n}{n} \right)^{\frac{s_fs_g}{2s_g + d_Z}}
    \right\}.\]
    \label{thm:active_upper_bound_appendix}
\end{customthm}

We now prove the result, utilizing several lemmas and key steps from the Passive upper bound in Appendix~\ref{app:passive_upper_bound}.

\begin{proof}
    As in the proof of the Passive Upper Bound (Theorem~\ref{thm:passive_upper_bound}), for any $X_1,...,X_n \in \X$, let
    \[\tilde f_{x_0,X_1,...,X_n,k} = \frac{1}{k} \sum_{i = 1}^k f(X_{\pi_i(x_0; \{X_1,...,X_n\})})\]
    denote the conditional expectation of $\hat f_{x_0,k}$ given $X_1,...,X_n$.
    By the triangle inequality,
    \begin{equation}
        \left| \hat f_{x_0,k} - f(x_0) \right|
          \leq \left| \hat f_{x_0,k} - \tilde f_{x_0,X_1,...,X_n,k} \right| + \left| \tilde f_{x_0,X_1,...,X_n,k} - f(x_0) \right|,
        \label{ineq:active_bias_variance_decomposition}
    \end{equation}
    where the first term captures stochastic error due to the response noise $\epsilon_Y$ and the second term captures smoothing bias.
    
    For the first term in Eq.~\eqref{ineq:active_bias_variance_decomposition}, as argued in the passive case (see Eq.~\eqref{ineq:active_upper_variance_bound}), since the response noise $\epsilon_Y$ is sub-Gaussian, with probability at least $1 - \delta/2$,
    \begin{equation}
        \left| \hat f_{x_0,k} - \tilde f_{x_0,X_1,...,X_n,k} \right|
        \leq \sigma_Y \sqrt{\frac{2}{k} \log \frac{4}{\delta}}.
        \label{ineq:active_upper_variance_bound}
    \end{equation}
    For the second term in Eq.~\eqref{ineq:active_bias_variance_decomposition}, by the triangle inequality and local H\"older continuity of $f$ around $x_0$, for some constant $C_1 > 0$,
    \begin{align}
        \notag
        \left| \tilde f_{x_0,X_1,...,X_n,k} - f(x_0) \right|
        & \leq \frac{1}{k} \sum_{i = 1}^k \left| f(X_{\pi_i(x_0; \{X_1,...,X_n\})}) - f(x_0) \right| \\
        \notag
        & \leq \frac{C_1}{k} \sum_{i = 1}^k \left\| X_{\pi_i(x_0; \{X_1,...,X_n\})} - x_0 \right\|^{s_f} \\
        \label{ineq:general_bias_smoothness_bound}
        & \leq C_1 \left\| X_{\pi_k(x_0; \{X_1,...,X_n\})} - x_0 \right\|^{s_f}.
    \end{align}
    
    By Lemmas~\ref{lemma:knn_distance} and \ref{lemma:pushforward_dimension}, for some constant $C_2 > 0$, with probability at least $1 - \delta/6$, if $\ell \geq 4 \log \frac{6}{\delta}$,
    \begin{equation}
        \|g(Z_{\pi_\ell(x_0, \{g(Z_1),...,g(Z_i)\})}) - x_0\| \leq C_2 \left( \frac{\ell}{m} \right)^{s_g/d_Z},
        \label{ineq:pilot_knn_distance}
    \end{equation}
    where $m = \lfloor n/2 \rfloor$ as in Algorithm~\ref{alg:active_learning_algorithm}.
    Meanwhile, by a standard maximal inequality for sub-Gaussian random variables (Lemma~\ref{lemma:sub_gaussian_maximal_inequality} in Appendix~\ref{app:supplementary_lemmas}), with probability at least $1 - \delta/6$,
    \begin{equation}
        \max_{i \in [\lfloor m/\ell \rfloor]} \left\| g(Z_i) - \bar X_i \right\| \leq \sigma_X \sqrt{\frac{2}{\ell} \log \frac{12m}{\delta}}.
        \label{ineq:maximal_pilot_error}
    \end{equation}
    By Inequality~\ref{ineq:maximal_pilot_error} and Inequality~\ref{ineq:pilot_knn_distance}, with probability at least $1 - 2\delta/6$,
    \begin{align*}
        \left\| g(Z_{\pi_\ell(x_0, \{\bar X_1,...,\bar X_i\})}) - x_0 \right\|
        & \leq \left\| \bar X_{\pi_\ell(x_0, \{\bar X_1,...,\bar X_{\lfloor m/\ell \rfloor}\})} - x_0 \right\| + \sigma_X \sqrt{\frac{2}{\ell} \log \frac{8m}{\delta}} \\
        & \leq \left\| \bar X_{\pi_\ell(x_0, \{\bar X_1,...,\bar X_{\lfloor m/\ell \rfloor}\})} - x_0 \right\| + \sigma_X \sqrt{\frac{2}{\ell} \log \frac{8m}{\delta}} \\
        & \leq C_2 \left( \frac{\ell}{m} \right)^{s_g/d_Z} + \sigma_X \sqrt{\frac{2}{\ell} \log \frac{8m}{\delta}}.
    \end{align*}
    Since $Z_{m+1} = \cdots = Z_n = Z_{\pi_\ell(x_0, \{\bar X_1,...,\bar X_{\lfloor m/\ell \rfloor}\})}$, by Lemma~\ref{lemma:knn_distance}, for some constant $C_3 > 0$, with probability at least $1 - \delta/6$, if $k \geq 4 \log \frac{6}{\delta}$,
    \[\left\| X_{\pi_k(x; \{X_{m+1},...,X_n\})} - g(Z_{\pi_\ell(x_0, \{\bar X_1,...,\bar X_{\lfloor m/\ell \rfloor}\})}) \right\|
      \leq C_3 \sigma_X \left( \frac{k}{n - m} \right)^{1/d_X},\]
    and it follows from the triangle inequality and the preceding inequality, with probability at least $1 - \delta/2$,
    \begin{align*}
        & \left\| X_{\pi_k(x_0; \{X_1,...,X_n\})} - x_0 \right\| \\
        & \leq \left\| X_{\pi_k(g(Z_{\pi_\ell(x_0, \{\bar X_1,...,\bar X_{\lfloor m/\ell \rfloor}\})}); \{X_{m+1},...,X_n\})} - g(Z_{\pi_\ell(x_0, \{\bar X_1,...,\bar X_{\lfloor m/\ell \rfloor}\})}) \right\| \\
        & + \|g(Z_{\pi_\ell(x_0, \{\bar X_1,...,\bar X_{\lfloor m/\ell \rfloor}\})}) - x_0\| \\
        & \leq C_3 \sigma_X \left( \frac{k}{n - m} \right)^{1/d_X}
          + C_2 \left( \frac{\ell}{m} \right)^{s_g/d_Z} + \sigma_X \sqrt{\frac{2}{\ell} \log \frac{12m}{\delta}}.
    \end{align*}
    Combining this with inequalities~\eqref{ineq:general_bias_smoothness_bound}, \eqref{ineq:active_upper_variance_bound}, and \eqref{ineq:active_bias_variance_decomposition} gives, with probability at least $1 - \delta$, as long as $k, \ell \geq 4 \log \frac{6}{\delta}$,
    \[\left| \hat f_{x_0,k} - f(x_0) \right|
        \leq C_4 \left( \sigma_Y \sqrt{\frac{1}{k} \log \frac{1}{\delta}}
        + \left( \sigma_X^{d_X} \frac{k}{n} \right)^{\frac{s_f}{d_X}}
        + \left( \frac{\ell}{n} \right)^{\frac{s_fs_g}{d_Z}} + \left( \sigma_X \sqrt{\frac{1}{\ell} \log \frac{n}{\delta}} \right)^{s_f} \right),\]
    for some constant $C_4 > 0$, since $m = \lfloor n/2 \rfloor$.
\end{proof}

\section{Lower Bounds for the Active Case (Theorem~\ref{thm:active_lower_bound})}
\label{app:active_lower_bounds}

In this section, we state and prove our main lower bounds for the Active case. As in the Passive case, our lower bounds will be given under global H\"older continuity (see Definition~\ref{def:global_holder}). We first state the main (overall) bound:

\begin{customthm}{\ref{thm:active_lower_bound}}[Main Active Lower Bound]
    In the active case, there exist constants $C, c > 0$, depending only on $L_g$, $s_g$, $L_f$, $s_f$, $d_X$, and $d_Z$, such that, for any estimator $\hat f_{x_0}$
    \[\inf_{\hat f_{x_0}} \sup_{g,f,P_Z,\epsilon_X,\epsilon_Y}
            \Pr_D \left[ \left| \hat f_{x_0} - f(x_0) \right| \geq C \Phi_{A,*} \right]
        \geq c,\]
    where
    \[\Phi_{A,*}
      = \max \left\{ \left( \frac{1}{n} \right)^{\frac{s_fs_g}{d_Z}},
        \left( \frac{\sigma_X^{d_X}}{n} \right)^{s_f/d_X},
        \frac{\sigma_Y}{\sqrt{n}},
        \left( \frac{\sigma_X^{d_X} \sigma_Y^2}{n} \right)^{\frac{s_f}{2s_f + d_X}} \right\}.\]
    \label{thm:active_lower_bound_app}
\end{customthm}

Since this lower bound is the maximum of four terms, it suffices to prove each of these four lower bounds separately. Recall that the $\asymp n^{-s_fs_g/d_Z}$ term was proven in Proposition~\ref{prop:lower_bound_no_noise}, which holds in both the Passive and Active cases. The remaining three lower bounds are prove in the the following three Propositions~\ref{prop:active_lower_bound_high_noise_case}, \ref{prop:active_lower_bound_no_response_noise}, and \ref{prop:active_lower_bound_no_covariate_noise}.

\begin{proposition}[Minimax Lower Bound, Active Case, $\sigma_X$ large, $\sigma_Y$ large]
    In the active case,
    \begin{equation}
        \inf_{\hat f_{x_0}} \sup_{g,f,P_Z,\epsilon_X,\epsilon_Y}
            \Pr_D \left[ \left| \hat f_{x_0} - f(x_0) \right|
                         \geq C \left( \frac{\sigma_X^{d_X} \sigma_Y^2}{n} \right)^{\frac{s_f}{2s_f + d_X}}
            \right]
        \geq \frac{1}{4},
    \end{equation}
    where $C > 0$ (given in Eq.~\eqref{eq:lower_bound_high_noise_constant}) depends only on $L_f$, $s_f$, and $d_X$.
    \label{prop:active_lower_bound_high_noise_case}
\end{proposition}

\begin{proof}
    As in the proofs of Propositions~\ref{prop:large_response_noise_passive_case} and \ref{prop:passive_lower_bound_no_response_noise}, consider the case in which $\sigma_X \epsilon_{X,i}$ is uniformly distributed on the cube $[-\frac{\sigma_X}{\sqrt{d_X}},\frac{\sigma_X}{\sqrt{d_X}}]^{d_X}$. Let
    \begin{equation}
        h := \left( \frac{\|K_{d_X}\|_{\C^{s_f}}^2}{4 d_X^{d_X/2} L_f^2} \frac{\sigma_X^{d_X} \sigma_Y^2}{n} \right)^{\frac{1}{2s_f + d_X}}.
        \label{eq:high_noise_h}
    \end{equation}
    Since $\epsilon_{X,i}$ has a uniform conditional distribution independent of $Z_i$ and $g$, the function
    \[x \mapsto \E_{D, g \sim P_0} \left[ 1\{\|x + \sigma_X \epsilon_{X,i}\|_\infty < h\} \right]\]
    is maximized over $x \in \R^{d_X}$ when $x = 0$. Note that, since this bound depends \emph{only} on the distribution of $\epsilon_{X,i}$, and not on the distribution of $Z_i$ or $g$, this applies even in the active case.
    \begin{align}
        \E_{D, g \sim P_0} \left[ 1\{\|X_i - x_0\|_\infty < h\} \right]
          \notag
          & = \E_{D, g \sim P_0} \left[ 1\{\|g(Z_i) + \sigma_X \epsilon_{X,i} - x_0\|_\infty < h\} \right] \\
          \label{ineq:uniform_density}
          & \leq \E_{D, g \sim P_0} \left[ 1\{\|\sigma_X \epsilon_{X,i}\|_\infty < h\} \right]
          = \min \left\{ 1, \left( \frac{h\sqrt{d_X}}{\sigma_X} \right)^{d_X} \right\}.
    \end{align}
    where $\phi_{d_X}$ denotes the $d_X$-dimensional standard normal density.
    Plugging this into Lemma~\ref{lemma:TV_upper_bound} gives
    \[\sup_{g,f,P_Z,\epsilon_X,\epsilon_Y} \Pr_D \left[ |\hat f_{x_0} - f(x_0)| \geq C_{d_X,s_f} L_f h^{s_f} \right]
        \geq \frac{1}{2} \left( 1 - \frac{L_f h^{s_f}}{\sigma_Y \|K_{d_X}\|_{\C^{s_f}}} \sqrt{n \left( \frac{h\sqrt{d_X}}{\sigma_X} \right)^{d_X}} \right).\]
    Plugging in our choice (Eq.~\eqref{eq:high_noise_h}) of $h$ gives
    \[\inf_{\hat F} \sup_{\theta \in \Theta} \Pr_{\theta} \left[ |\hat F - F(\theta)| \geq s \right]
      \geq \frac{1}{4},\]
    for
    \begin{equation}
        s = \frac{1}{2} \left( \frac{L_f}{\|K_{d_X}\|_{\C^{s_f}}} \right)^{\frac{d_X}{2s_f + d_X}}
    \left( 4 \sqrt{d_X^{d_X}} \right)^{\frac{-s_f}{2s_f + d_X}} \left( \frac{\sigma_X^{d_X} \sigma_Y^2}{n} \right)^{\frac{s_f}{2s_f + d_X}}.
        \label{eq:lower_bound_high_noise_constant}
    \end{equation}
\end{proof}

Next, we prove a tight lower bound for the case that the covariate noise level $\sigma_X$ is large but the response noise level $\sigma_Y$ is small. This proof is quite similar to that of Proposition~\ref{prop:passive_lower_bound_no_response_noise} for the passive case, except that, (a) rather that considering $Z$ uniformly distributed on $\Z$, we make no assumptions on $Z$ and utilize only randomness coming from $\epsilon_X$, and (b) we must reason somewhat more carefully about the interdependence between samples.

\begin{proposition}[Minimax Lower Bound, Active Case, $\sigma_X$ large, $\sigma_Y$ small]
    In the active case,
    \begin{equation}
        \inf_{\hat f_{x_0}} \sup_{g,f,P_Z,\epsilon_X,\epsilon_Y}
            \Pr_D \left[ \left| \hat f_{x_0} - f(x_0) \right|
                         \geq C
                              \left( \frac{\sigma_X^{d_X}}{n} \right)^{s_f/d_X}
            \right]
        \geq \frac{1}{2e} \approx 0.18,
    \end{equation}
    where $C > 0$ (given below in Eq.~\eqref{eq:active_lower_bound_no_response_noise_constant}) depends only on $L_g$, $s_g$, $L_f$, $s_f$, $d_X$, and $d_Z$.
    \label{prop:active_lower_bound_no_response_noise}
\end{proposition}

\begin{proof}
    Suppose $\sigma_X \epsilon_{X,i}$ is uniformly distributed on $[-\frac{\sigma_X}{\sqrt{d_X}},\frac{\sigma_X}{\sqrt{d_X}}]^{d_X} \subseteq \R^{d_X + 1}$. It is straightforward to verify that $\epsilon_{X,i}$ is sub-Gaussian with and that $\epsilon_{X,i}$ has dimension $d_X$ around $0$.
    Let $r^* := \frac{\sigma_X}{\sqrt{d_X}}$. For any $r \in [0, r^*]$, since each $\epsilon_i$ is independent of the preceding data $\{(X_j,Y_j,Z_j)\}_{j = 1}^n$, and since the function $x \mapsto \Pr \left[ \|x + \epsilon_i\|_\infty \leq r \right]$ is maximized at $x = 0$,
    \begin{align*}
        \Pr \left[ \min_{i \in [n]} \|X_i\|_\infty \geq r \right]
        & = \prod_{i = 1}^n \Pr \left[ \|X_i\|_\infty \geq r \middle| \min_{j \in [i - 1]} \|X_j\|_\infty \geq r \right] \\
        & = \prod_{i = 1}^n \left( 1 - \Pr \left[ \|X_i\|_\infty \leq r \middle| \min_{j \in [i - 1]} \|X_j\|_\infty \geq r \right] \right) \\
        & = \prod_{i = 1}^n \left( 1 - \Pr \left[ \|g(Z_i) + \epsilon_i\|_\infty \leq r \middle| \min_{j \in [i - 1]} \|X_j\|_\infty \geq r \right] \right) \\
        & \geq \prod_{i = 1}^n \left( 1 - \Pr \left[ \|\epsilon_i\|_\infty \leq r \middle| \min_{j \in [i - 1]} \|X_j\|_\infty \geq r \right] \right) \\
        & = \prod_{i = 1}^n \left( 1 - \Pr \left[ \|\epsilon_i\|_\infty \leq r \right] \right) \\
        & = \left( 1 - \left( \frac{r \sqrt{d_X}}{\sigma_X} \right)^{d_X} \right)^n.
    \end{align*}
    For $n$ sufficiently large that $r \leq r^*$, plugging in
    \begin{equation}
        r = \frac{\sigma_X}{\sqrt{d_X}} n^{-1/d_X}.
        \label{eq:min_nearest_neighbor_radius_active_case}
    \end{equation}
    gives
    \[\Pr \left[ \min_{i \in [n]} \|X_i\|_\infty \geq \frac{\sigma_X}{\sqrt{d_X}} n^{-1/d_X} \right]
        = \left( 1 - \frac{1}{n} \right)^n
        \geq 1/e > 0,\]
    since $\left( 1 - \frac{1}{n} \right)^n$ approaches $1/e$ from above as $n \to \infty$.
    
    Now let $f_0$ and $f_1$ be as in Eq.~\eqref{eq:f0_and_f1}, with $h = r$ above. Given the event $\min_{i \in [n]} \|X_i\|_\infty \geq h$, the distribution of the data $\{(X_i,Y_i,Z_i)\}_{i = 1}^n$ is independent of whether $f = f_0$ or $f = f_1$. Hence, as argued in the proof of Proposition~\eqref{prop:lower_bound_no_noise},
    \[\max_{f \in \{f_0,f_1\}} \Pr \left[ \left| \hat f_{x_0} - f(x_0) \right| \geq \frac{|f_1(x_0) - f_0(x)|}{2} \right]
        \geq \frac{1}{2} \Pr \left[ \min_{i \in [n]} \|X_i\|_\infty \leq h \right]
        \geq \frac{1}{2e}.\]
    Since, by construction of $f_0$ and $f_1$,
    \[|f_1(x_0) - f_0(x_0)| = \frac{L_f h^{s_f}}{\|K_{d_X}\|_{\C^{s_f}}},\]
    it follows that
    \[\max_{f \in \{f_0,f_1\}} \Pr \left[ \left| \hat f_{x_0} - f(x_0) \right| \geq \frac{L_f h^{s_f}}{2\|K_{d_X}\|_{\C^{s_f}}} \right]
        \geq \frac{1}{2e}.\]
    Plugging in our choice of $h = r$ from Eq~\eqref{eq:min_nearest_neighbor_radius_active_case} above gives
    \[\max_{f \in \{f_0,f_1\}} \Pr \left[ \left| \hat f_{x_0} - f(x_0) \right| \geq C \left( \frac{\sigma_X^{d_X}}{n} \right)^{s_f/d_X} \right]
        \geq \frac{1}{2e},\]
    where
    \begin{equation}
        C =  \frac{L_f d_X^{-s_f/2}}{2\|K_{d_X}\|_{\C^{s_f}}}.
        \label{eq:active_lower_bound_no_response_noise_constant}
    \end{equation}
\end{proof}

Finally, we prove a lower bound for the Active case that is tight when $\sigma_X$ is small but $\sigma_Y$ is large:

\begin{proposition}[Minimax Lower Bound, Active Case, $\sigma_X$ small, $\sigma_Y$ large]
    In the active case,
    \begin{equation}
        \inf_{\hat f_{x_0}} \sup_{g,f,P_Z,\epsilon_X,\epsilon_Y}
            \Pr_D \left[ \left| \hat f_{x_0} - f(x_0) \right|
                         \geq \frac{\sigma_Y}{4\sqrt{n}}
            \right]
        \geq \frac{1}{4}.
    \end{equation}
    \label{prop:active_lower_bound_no_covariate_noise}
\end{proposition}

\begin{proof}
    Plugging the trivial bound
    \[\sum_{i = 1}^n \E_{D, g \sim P_0} \left[ 1\{\|X_i - x_0\|_\infty < h\} \right]
      \leq n
      \quad \text{ and } \quad
      h = \left( \frac{\sigma_Y \|K_{d_X}\|_{\C^{s_f}}}{2L_f\sqrt{n}} \right)^{1/s_f}\]
    into Lemma~\ref{lemma:TV_upper_bound} gives
    \[\inf_{\hat F} \sup_{\theta \in \Theta} \Pr_{\theta} \left[ |\hat F - F(\theta)| \geq \frac{\sigma_Y}{4\sqrt{n}} \right]
      \geq \frac{1}{4}.\]
\end{proof}

\section{Supplementary Lemmas}
\label{app:supplementary_lemmas}

For the reader's convenience, here we state two standard concentration inequalities for sub-Gaussian random variables, used in our upper bounds.
\begin{lemma}[Theorem 1.14 of \citet{rigollet201518}]
    Suppose that $X_1,...,X_n$ are IID observations of a sub-Gaussian random variable. Then, for any $\delta \in (0, 1)$,
    \[\Pr_{X_1,...,X_n} \left[ \max_{i \in [n]} |X_i| \leq \sqrt{2 \log \left( \frac{2n}{\delta} \right)} \right] \leq \delta.\]
    \label{lemma:sub_gaussian_maximal_inequality}
\end{lemma}

\begin{lemma}[Proposition 2.1 of \citet{wainwright2019high}]
    Suppose that $X_1,...,X_n$ are IID observations of a sub-Gaussian random variable. Then, for any $\delta \in (0, 1)$,
    \[\Pr_{X_1,...,X_n} \left[ \frac{1}{n} \sum_{i = 1}^n X_i \leq \sqrt{\frac{2}{n} \log \frac{1}{\delta}} \right] \leq \delta.\]
    \label{lemma:sub_gaussian_mean_concentration}
\end{lemma}

\section{Details of Experimental Results}
\label{app:experiments}

This section provides further details regarding the experimental results reported in the main paper. Python code and instructions for reproducing Figures~\ref{fig:experimental_results}, \ref{fig:SIR_performance}, and \ref{fig:infected_plots} of the main paper are available at \url{https://gitlab.tuebingen.mpg.de/shashank/indirect-active-learning}. Experiments we run in Python 3.9, on an Ubuntu laptop with Intel Core i7-10750H CPU.

\paragraph{Code}
In the code included in the supplementary material:
\begin{itemize}
    \item \texttt{README.md} explains how to set up and run the code.
    \item \texttt{knn\_regressors.py} implements the Passive, Passive CV, Active, Active CV, Oracle, and Oracle CV estimators used in Experiments 1 and 2 of the main paper.
\end{itemize}

\subsection{Synthetic Data}
\label{app:synthetic_data}

\paragraph{Code}
In the code included in the supplementary material:
\begin{itemize}
    \item \texttt{sampler.py} produces joint samples of $(X, Y, Z)$, as well as conditional samples of $(X, Y, Z)$ given $Z$.
    \item \texttt{figure\_3.py} carries out Experiment 1 described in Section~\ref{subsec:synthetic_data} of the main paper. It also generates Figure~\ref{fig:experimental_results}.
\end{itemize}

\subsection{Epidemiological Forecasting Application}
\label{app:SIRD}

\paragraph{SIRD Simulation}
The stochastic SIRD simulation was run for $T = 100$ time steps over a total population of $N = 1000$, beginning with $10$ initially infected individuals ($I_0 = 10$), $990$ susceptible individuals ($S_0 = 990$), and no deceased or recovered individuals ($R_0 = D_0 = 0$). The simulation takes in three $[0, 1]$-valued hyperparameters, an infection rate $\beta$, a recovery rate $\gamma$, and a death rate $\delta$.

Given these inputs and initial conditions, the stochastic SIRD simulation proceeds as follows.
At each timepoint $t$, a random number of individuals from the Susceptible compartment is newly infected, according to
\[S_{t - 1} - S_t \sim \mathsf{Binomial}(S_{t - 1}, \beta I/N),\]
where the rate $\beta I/N$ of new infections is proportional to the proportion $I/N$ of infected individuals in the population. Meanwhile, members of the Infected compartment either (i) recover, (ii) die, or (iii) stay infected according to
\[(R_t - R_{t - 1}, D_t - D_{t - 1}, \text{---} ) \sim \mathsf{Multinomial}(I_{t - 1}, (\gamma, \delta, 1 - (\gamma + \delta)))\]
distribution. The new number of infected individuals is then updated according to
\[I_t = N - (S_t + R_t + D_t).\]
This process was iterated for $t \in \{1,...,T\}$.

\paragraph{Hyperparameters Distributions}
The first component of the distribution $P_Z$ on $[0, 1]^3$ was a $\mathsf{Beta}(2, 8)$ distribution, while the second and third components were the first two components of a $\mathsf{Dirichlet}(1.0, 0.3, 8.7)$ distribution, reflecting plausible ranges for these parameters in recent COVID-19 outbreaks~\citep{calafiore2020time,bousquet2022deep}.

\paragraph{Code}
In the code included in the supplementary material:
\begin{itemize}
    \item \texttt{SIR\_simulation.py} produces joint samples of $(X, Y, Z)$, as well as conditional samples of $(X, Y, Z)$ given $Z$.
    \item \texttt{figure\_4.py} carries out Experiment 2 described in Section~\ref{subsec:SIRD} of the main paper. It also generates Figure~\ref{fig:SIR_performance}.
    \item \texttt{figure\_5.py} generates Figure~\ref{fig:infected_plots} of the main paper.
\end{itemize}

\end{document}